\documentclass[11pt,a4paper]{article}
\usepackage{geometry}
\usepackage[english]{babel}
\usepackage{amssymb}
\usepackage{enumitem}
\usepackage{latexsym}
\usepackage[all,cmtip]{xy}
\usepackage{amsmath,amsthm}
\usepackage{comment}
\usepackage{tikz}
\usepackage{marvosym}
\usepackage{tabu}
\usepackage{hyperref}
\usepackage{color}

\usepackage{etoolbox}
 
\makeatletter
\patchcmd{\maketitle}{\@fnsymbol}{\@alph}{}{}  
\makeatother
 
\title{Morita equivalence classes of blocks with elementary abelian defect groups of order $32$}
\author{
   Cesare Giulio Ardito\thanks{Department of Mathematics, City University of London, Northampton Square, London, EC1V 0HB, United Kingdom. Email: cesareg.ardito@gmail.com}
}
\date{\today}

\newtheorem{theorem}{Theorem}[section]
\newtheorem*{theorem*}{Theorem}

\newtheorem{method}[theorem]{Method}
\newtheorem{lemma}[theorem]{Lemma}
\newtheorem{corollary}[theorem]{Corollary}
\newtheorem{proposition}[theorem]{Proposition}
\newtheorem{remark}[theorem]{Remark}

\newcommand{\Irr}{\operatorname{Irr}}

\newcommand{\Hom}{\mathop{\rm Hom}\nolimits}

\newcommand{\Aut}{\operatorname{Aut}}
\newcommand{\Out}{\operatorname{Out}}

\newcommand{\Pic}{\mathop{\rm Pic}\nolimits}

\newcommand{\SL}{\operatorname{SL}}

\newcommand{\GL}{\operatorname{GL}}

\newcommand{\PGL}{\rm PGL}
\newcommand{\PSL}{\rm PSL}

\newcommand{\cO} {\mathcal{O}}
\newcommand{\cT} {\mathcal{T}}

\newcommand{\cE} {\mathcal{E}}

\newcommand{\K} {K}

\def\bigcp{\mathop{\mathchoice 
 {\hbox{\sf\Large\lower 0.1\baselineskip\hbox{Y}}}%
 {\hbox{\sf\large\lower 0.1\baselineskip\hbox{Y}}}%
 {\hbox{\sf\normalsize\lower 0.1\baselineskip\hbox{Y}}}%
 {\hbox{\sf\tiny\lower 0.1\baselineskip\hbox{Y}}}%
}}
\def\bigtimes{\mathop{\mathchoice 
 {\hbox{\sf\Large\lower 0.1\baselineskip\hbox{X}}}%
 {\hbox{\sf\large\lower 0.1\baselineskip\hbox{X}}}%
 {\hbox{\sf\normalsize\lower 0.1\baselineskip\hbox{X}}}%
 {\hbox{\sf\tiny\lower 0.1\baselineskip\hbox{X}}}%
}}

\def\Sym(#1){\mathop{\rm Sym}(#1)}
\def\Sym(#1){S_{#1}}
\def\diag(#1){\mathop{\rm diag}(#1)}

\begin{document}\setcounter{MaxMatrixCols}{50}
\maketitle \vspace*{-1em}
\begin{abstract}
We describe a general technique to classify blocks of finite groups, and we apply it to determine Morita equivalence classes of blocks with elementary abelian defect groups of order $32$ with respect to a complete discrete valuation ring with an algebraically closed residue field of characteristic two. As a consequence we verify that a conjecture of Harada holds on these blocks. 

Keywords: Donovan's conjecture; Finite groups; Morita equivalence; Block theory; Modular representation theory.

2010 Mathematics Subject Classification: Primary 20C20; Secondary 16D90, 20C05.
\end{abstract}

\normalsize

\section{Introduction}
Let $\cO$ be a complete discrete valuation ring, with field of fractions $K$ of characteristic $0$ and residue field $k$, an algebraically closed field of characteristic $p$. Note that $K$ and $k$ cannot both be algebraically closed, but we can assume $K$ to be large enough for all finite groups considered in this paper. The triple $(K,\cO, k)$ is usually called a \emph{$p$-modular system}.

We say that two algebras $A$ and $B$ are Morita equivalent if there is an equivalence between the categories of $A$-modules and $B$-modules. Given a finite group $G$, we consider blocks of the group algebras $\cO G$ and $kG$. Given a block $B$ of $\cO G$, there is a unique correspondent block $\overline{B}$ of $kG$ via the canonical map $B \mapsto \overline{B} = B \otimes_\cO k$. Moreover, a Morita equivalence between blocks of group algebras over $\cO$ implies a Morita equivalence between the same blocks of the group algebras over $k$, while the converse is not known to be true. Hence, classifying blocks over $\cO$ is, in general, harder than classifying blocks over $k$. 

For a $p$-block $B$ of $\cO G$ we consider the defect group $D$, a $p$-subgroup of $G$ defined up to conjugation as the unique maximal group among the vertices of the modules in the block $B$. A $B$-subpair is a pair $(Q,B_Q)$ where $Q$ is a subgroup of $D$ and $B_Q$ is a block of $\cO Q C_G(Q)$ with Brauer correspondent $B$. When $Q=D$, the $B$-subpairs $(D,B_D)$ are all $G$-conjugate and we write $e:=B_D$, often called a \textit{root} of $B$. The inertial quotient of $B$ is defined as $E=N_G(D,e)/C_G(D)$, where $N_G(D,e)$ denotes the stabilizer of $e$ in $N_G(D)$. $E$ is always a $p'$-group and, when $D$ is abelian, $B$ is nilpotent if and only if $E=1$. 

We denote the number of irreducible characters of $KG$ in the block $B$ as $k(B)$, and the number of irreducible Brauer characters of $kG$ in the block $B$ as $l(B)$. Moreover, we denote by $\mathcal{F}$ the fusion system of $D$ given by the block $B$ (see \cite[8.1]{lin18}). For a group algebra $\cO G$, we denote as $B_0(\cO G)$ the principal block, i.e. the one that contains the trivial character. 

Donovan's conjecture states that for each isomorphism class of a $p$-group $D$ there is a finite number of Morita equivalence classes of blocks of finite groups with defect group $D$. However, note that defect groups are not known to be invariant under Morita equivalence, and inertial quotients, in general, are not invariant \cite{dade71}. Donovan's conjecture has been proved over $k$ for elementary abelian $2$-groups in \cite{ekks}, later generalized to abelian $2$-groups in \cite{eali18c} over $k$ and in \cite{ealiei18} over $\cO$. However, the proofs do not produce an explicit list of all the classes for each fixed defect group. Our purpose is to describe the Morita equivalence classes of blocks with defect group $(C_2)^5$.

\begin{theorem}
\label{maintheorem}
Let $G$ be a finite group, and let $B$ be a block of $\cO G$ with elementary abelian defect group $D$ of order $32$. Then one of the following two possibilities occurs: \begin{itemize}

\item[$\star$)] $B$ is Morita equivalent to the principal block of precisely one of the following groups: \vskip0.5em
\begin{tabular}{rlll}
(i) & 	$(C_2)^5$ &\qquad \qquad & (inertial quotient $\{1\}$) \\
(ii) & 	$A_4 \times (C_2)^3$  &&  (i.q. $(C_3)_1$) \\
(iii)&	$A_5 \times (C_2)^3$  &&  (i.q. $(C_3)_1$) \\
(iv) & 	$((C_2)^4 \rtimes C_3) \times C_2$  &&  (i.q. $(C_3)_2$) \\
(v) & 	$((C_2)^4 \rtimes C_5) \times C_2$  \\ 
(vi) & 	$((C_2)^3 \rtimes C_7) \times (C_2)^2$  &&  (i.q. $C_7$) \\
(vii) & $SL_2(8) \times (C_2)^2$   &&  (i.q. $C_7$) \\
(viii)  & $A_4 \times A_4 \times C_2$   &&  (i.q. $C_3 \times C_3$) \\
(ix) &	$A_4 \times A_5 \times C_2$    &&  (i.q. $C_3 \times C_3$) \\
(x) &	$A_5 \times A_5 \times C_2$    &&  (i.q. $C_3 \times C_3$) \\
(xi)&	$((C_2)^4 \rtimes C_{15}) \times C_2$  \\  
(xii) & $SL_2(16) \times C_2$  \\  
(xiii) & $((C_2)^3 \rtimes C_7) \times A_4$   &&  (i.q. $C_{21}$)\\
(xiv)&	 $((C_2)^3 \rtimes C_7) \times A_5$   &&  (i.q. $C_{21}$) \\
(xv)&	 $SL_2(8) \times A_4$   &&  (i.q. $C_{21}$) \\
(xvi)&	 $SL_2(8) \times A_5$    &&  (i.q. $C_{21}$) \\
(xvii)&	 $((C_2)^3 \rtimes (C_7 \rtimes C_3)) \times (C_2)^2$   &&  (i.q. $(C_7 \rtimes C_3)_1$) \\
(xviii)& $J_1 \times (C_2)^2$    &&  (i.q. $(C_7 \rtimes C_3)_1$) \\
(xix)&	 $\Aut(SL_2(8)) \times (C_2)^2$   &&  (i.q. $(C_7 \rtimes C_3)_1$) \\
(xx) &	 $(C_2)^5 \rtimes (C_7 \rtimes C_3)$   &&  (i.q. $(C_7 \rtimes C_3)_2$) \\
(xxi)&	 $(SL_2(8) \times (C_2)^2) \rtimes C_3$   &&  (i.q. $(C_7 \rtimes C_3)_2$) \\
(xxii)&	 $(C_2)^5 \rtimes C_{31}$   &&  (i.q. $C_{31}$) \\
(xxiii)& $SL_2(32)$  &&  (i.q. $C_{31}$) \\
(xxiv)& $((C_2)^3 \rtimes (C_7 \rtimes C_3)) \times A_4$   &&  (i.q. $(C_7 \rtimes C_3) \times C_3$) \\
(xxv)& $((C_2)^3 \rtimes (C_7 \rtimes C_3)) \times A_5$   &&  (i.q. $(C_7 \rtimes C_3) \times C_3$) \\
(xxvi) & $J_1 \times A_4$    &&  (i.q. $(C_7 \rtimes C_3) \times C_3$)\\
(xxvii) & $J_1 \times A_5$   &&  (i.q. $(C_7 \rtimes C_3) \times C_3$) \\
(xxviii)& $\Aut(SL_2(8)) \times A_4$   &&  (i.q. $(C_7 \rtimes C_3) \times C_3$) \\
(xxix) & $\Aut(SL_2(8)) \times A_5$   &&  (i.q. $(C_7 \rtimes C_3) \times C_3$) \\
(xxx) & $(C_2)^5 \rtimes (C_{31} \rtimes C_5)$    &&  (i.q. $C_{31} \rtimes C_5$) \\
(xxxi) & $\Aut(SL_2(32))$  &\qquad \qquad&  (i.q. $C_{31} \rtimes C_5$)
\end{tabular} \vskip1em
\noindent 
\item[$\star\star$)] $B$ is Morita equivalent to a nonprincipal block of one of the following groups (there is exactly one such Morita equivalence class for each group): \vskip0.5em
\begin{tabular}{rlll}
\qquad (a) & $((C_2)^4 \rtimes 3_+^{1+2}) \times C_2$  &\qquad \qquad & (i.q. $C_3 \times C_3$) \\
(b) & $(C_2)^5 \rtimes (C_7 \rtimes 3_+^{1+2})$ && (i.q. $(C_7 \rtimes C_3) \times C_3$) \\
(c) & $(SL_2(8) \times (C_2)^2) \rtimes  3_+^{1+2}$  && (i.q. $(C_7 \rtimes C_3) \times C_3$)
\end{tabular} \end{itemize}
\noindent Moreover, if a block $C$ of $\cO H$ for a finite group $H$ is Morita equivalent to $B$, then the defect group of $C$ is isomorphic to $D$.
\end{theorem} 

In Section 2, we list preliminary results and some reductions that we use in the proof of the main theorem. In Section 3 we look at perfect isometries between certain blocks, which we use to extend the main theorem of \cite{kk96} over $\cO$. This allows us to examine blocks of groups that cover a block of a normal subgroup with index a power of two. In Section 4 we give background on crossed products and Picard groups, which allow us to examine blocks of groups that cover a block of a normal subgroup with index an odd prime, and apply this method to study some cases that arise when proving our main result. In Section 5 we prove our main theorem, list all the classes and investigate their invariants, and then verify a conjecture of Harada for these blocks.

Detailed information about each class of blocks is available on the Block Library \cite{ea20}.

\section{Reductions and technical lemmas}

Given a normal subgroup $N \lhd G$, $B$ a block of $\cO G$ and $b$ a block of $\cO N$, we say that $B$ \textit{covers} $b$ when $Bb \neq 0$. The structures of $B$ and $b$ are closely related in this case. For instance, as shown in \cite[15.1]{alp151}, a defect group of $b$ is the intersection of a defect group of $B$ with $N$. This relation is the main tool that we use to obtain our classification.

As we mentioned in the introduction, two blocks $B$ and $C$ of finite groups are Morita equivalent if their module categories are equivalent. An alternative, more explicit way to define this equivalence is to say that there is a $(B,C)$-bimodule $M$ and a $(C,B)$-bimodule $N$ such that $M \otimes_C N \cong B$ as $(B,B)$-bimodules and $N \otimes_B M \cong C$ as $(C,C)$-bimodules. We say that the Morita equivalence is \emph{realised} by $M$ and $N$. Two blocks are \emph{basic Morita equivalent} if they are Morita equivalent via an equivalence realised by bimodules with endopermutation source. Two blocks are \emph{source algebra equivalent} (or \emph{Puig equivalent}) if they are Morita equivalent via an equivalence realised by bimodules with trivial source. Note that each equivalence is stronger than the ones above it.

Given a finite group $G$, a block $B$ of $\cO G$ is said to be quasiprimitive if for any normal subgroup $N \lhd G$ there is a unique block $b$ of $\cO N$ covered by $B$. This is equivalent, by \cite[15.1]{alp151}, to the requirement that $b$ is $G$-stable under the action of $G$ by conjugation on $\cO N$. We can reduce to quasiprimitive blocks in many situations using the Fong-Reynolds correspondence:

\begin{theorem}[6.8.3 \cite{lin18}] \label{fong1}
Let $G$ be a finite group, and let $N$ be a normal subgroup of $G$. Let $b$ be a block of $\cO N$ and $B$ be a block of $\cO G$  that covers $b$. Let $H$ be the stabiliser of $b$ in $G$ acting by conjugation. Then there is a one-to-one correspondence between the blocks of $\cO G$ that cover $b$ and the blocks of $\cO H$ that cover $b$, where each block is source algebra equivalent to its correspondent.
\end{theorem}

We also use the following version of Fong's Second Reduction, used whenever $B$ covers a nilpotent block, and in particular when $G$ has a normal $p'$-subgroup.

\begin{theorem}[\cite{kupu90}] 
Let $G$ be a finite group and $N \lhd G$. Let $B$ be a block of $\cO G$ with defect group $D$ that covers a $G$-stable nilpotent block $b$ of $\cO N$. Then there are finite groups $M \lhd L$ such that $M \cong D \cap N$, $L/M \cong G/N$, there is a subgroup $D_L \leq L$ with $D_L \cong D$ and $M \leq D_L$, and there is a central extension $\widehat{L}$ of $L$ by a $p'$-group, and a block $\widehat{B}$ of $\cO \tilde{L}$ which is Morita equivalent to $B$ and has defect group $\widehat{D} \cong D$. If $B$ is the principal block, then so is $\widehat{B}$.
\end{theorem}
In particular we can assume that $Z(\widehat{L}) \leq [\widehat{L},\widehat{L}]$, and hence that $\widehat{L}$ is contained in a Schur representation group \cite[\S 11]{isaacs} of $L$, as otherwise we could consider a smaller group that still has a block Morita equivalent to $B$ (see also \cite[6.8.13]{lin18}).
 
\begin{corollary}[\cite{ea17}] \label{kulshammerpuigcorollary}
Let $G$ be a finite group, let $N \lhd G$ with $N \not\leq Z(G)O_p(G)$. Let $B$ be a quasiprimitive $p$-block of $\cO G$ covering a nilpotent block of $\cO N$. Then there is a finite group $H$ with $[H:O_{p'}(Z(H))] < [G:O_{p'}(Z(G))]$ and a block $B_H$ with isomorphic defect group to the one of $B$, such that $B_H$ is Morita equivalent to $B$.
\end{corollary}

Given a block $B$ of $\cO G$, a normal subgroup $N \lhd G$ and a block $b$ of $\cO N$ covered by $B$, and given a chain of normal subgroups $N = N_0 \lhd N_1 \lhd \dots \lhd N_t = G$ we define a \textit{block chain} to be any sequence of blocks $b_i$ of $\cO N_i$ such that $b_{i+1}$ covers $b_i$, $b_0=b$ and $b_t=B$. 
\begin{lemma} \label{pcore}
Let $G$ be a finite group and $B$ a quasiprimitive block of $\cO G$ with a defect group $D$ of order $p^n$. Let $N$ be a normal subgroup of $G$ and let $b$ be a block of $\cO N$ covered by $N$. If $G/N$ is solvable, then $DN/N$ is a Sylow $p$-subgroup of $G/N$.
\end{lemma}
\begin{proof}
Since $G/N$ is solvable, we can consider the upper $p$-series of $G/N$ \cite[\S 6.3]{gor07}:
$$1 \lhd O_p(G/N) \lhd O_{p,p'}(G/N) \lhd O_{p,p',p}(G/N) \lhd \dots \lhd G/N$$
Every group in the chain is a characteristic subgroup of $G/N$, and each index is either a power of $p$ or prime to $p$. We can take a preimage of this series under $\pi:G \to G/N$ and obtain:
$$N_0 = N \lhd N_1 \lhd N_2 \lhd N_3 \lhd \dots \lhd N_t = G$$
Consider the corresponding block chain given by the unique blocks $b_i$ of $\cO N_i$ covered by $B$, and let $D_i$ be the defect group of $b_i$. We distinguish two cases: \begin{itemize}
\item $[N_{i+1}:N_i]$ is prime to $p$. Then $b_i$ and $b_{i+1}$ share a defect group, so $D_i = D_{i+1}$.
\item $[N_{i+1}:N_i]$ is a power of $p$. Then $b_i$ is $N_{i+1}$ stable because it is $G$-stable, and $b_{i+1}$ is the unique block of $\cO N_{i+1}$ covering $b_i$ \cite[5.3.5]{feit}. Then by \cite[15.1]{alp151} we have that $[D_{i+1} : D_i] = [N_{i+1}:N_i]$.
\end{itemize}
Then:
$$[D:D_0] = [D:D_{t-1}]\dots[D_1:D_0] = [G:N_{t-1}]_p\dots[N_1:N_0]_p = |G/N|_p$$
Since from \cite[15.1]{alp151} $D_0 = D \cap N$, we are done. 
\end{proof}

A block $B$ of $\cO G$ is \textit{nilpotent covered} if there exists a group $\tilde{G} \rhd G$ and a nilpotent block $\tilde{B}$ of $\cO \tilde{G}$ such that $\tilde{B}$ covers $B$. We say that $B$ is \textit{inertial} if it is basic Morita equivalent to its Brauer correspondent. The following lemma relates these two concepts.
\begin{lemma}[\cite{pu11}, \cite{zhou16}]  \label{nilpotentcovered}
Let $G$ be a finite group and let $N \lhd G$. Let $b$ be a $p$-block of $\cO N$ covered by a block $B$ of $\cO G$. Then: 
\begin{enumerate}[label=(\arabic*)] \setlength\itemsep{0em}
\item If $B$ is inertial, then $b$ is inertial.
\item If $b$ is nilpotent covered, then $b$ is inertial.
\item If $p$ does not divide $[G:N]$ and $b$ is inertial, then $B$ is inertial.
\item If $b$ is nilpotent covered, then it has abelian inertial quotient.
\end{enumerate}
\end{lemma}
\begin{proof}(1), (2) and (4) are respectively Theorem 3.13 and Corollary 4.3 in \cite{pu11}. (3) is the main theorem of \cite{zhou16}.
\end{proof}

Given two finite groups $N \lhd G$ and a block $b$ of $\cO N$, we define $G[b]$ as the group of elements of $G$ acting as inner algebra automorphisms on $b \otimes_\cO k$. We use the following result, extracted from \cite{kekoli}, when dealing with automorphisms of products of quasisimple groups. 
\begin{lemma}[\cite{kekoli}, \cite{ea14}] \label{gbstarlemma}
Let $B$ be a block of $\cO G$ with defect group $D$, and let $N \lhd G$ such that $D \leq N$, and $b$ a $G$-stable block of $\cO N$ covered by $B$. Let $\hat{b}$ be a block of $G[b]$ covered by $B$. Then $b$ is source algebra equivalent to $\hat{b}$, and $B$ is the unique block of $\cO G$ that covers $\hat{b}$. In particular, if $[G:N]=\ell$ for a prime $\ell \neq p$, either $B$ is the unique block that covers $b$ or $B$ is source algebra equivalent to $b$.
\end{lemma}
\begin{proof}
Proposition 2.2 in \cite{kekoli} gives the source algebra equivalence between $\hat{b} \otimes_\cO k$ and $b \otimes_\cO k$. By Proposition 7.8 in \cite{pu88b}, this equivalence lifts to $\hat{b}$ and $b$. The uniqueness of $B$ follows from Theorem 3.5 in \cite{murai12}.
\end{proof}

Given two groups $G_1$ and $G_2$, every block of $\cO(G_1 \times G_2)$ is of the form $b_1 \otimes b_2$, where $b_i$ is a block of $\cO G_i$. Ordinary representation theory of central products is very similar to that of direct products, a fact that still holds for $p$-blocks as long as the shared center is a $p'$-group.
\begin{lemma}[7.5 \cite{sam14}, 1.5 \cite{duvel04}] \label{centralproducts}
Let $G=G_1 * G_2$ be a central product of finite groups $G_1$ and $G_2$, and let $B$ be a $p$-block of $\cO G$ with defect group $D$, and $B_i$ the unique block of $\cO G_i$ covered by $B$, with defect group $D_i$. Then:
\begin{enumerate}[label=(\arabic*)]\setlength\itemsep{0em}
\item $D=D_1D_2 = D_1*D_2$ is a defect group of $B$.
\item If $G_1 \cap G_2$ has $p'$ order, then $B$ is isomorphic to $B_1 \otimes B_2$.
\item $B$ is nilpotent if and only if $B_i$ is nilpotent for each $i$.
\end{enumerate}
\end{lemma}
\noindent Note that this lemma can be generalized to a central product of any number of groups.\\

From now on, $p=2$ and $D\cong (C_2)^5$. Blocks with cyclic defect group $C_2$ are classified in \cite{alp151}. The classification for Klein four defect group appears, for instance, in \cite{erd82}, \cite{lin94}, \cite{cra11}. The classifications for $(C_2)^3$ and $(C_2)^4$ are the main theorems of \cite{ea14} and \cite{ea17} respectively.

Numerical invariants for blocks with defect group $(C_2)^5$ have not been completely determined before this work. Nevertheless, using Brauer's second main theorem and a result from \cite{sam14} we can obtain a list of possibilities for these values, that we use to deal with certain situations in the proof of the main theorem. At the end of this paper we will know exactly what cases occur (see Proposition \ref{truesambalecalcs}).
\begin{proposition} \label{sambalecalcs}
Let $B$ be a block of $\cO G$ where $G$ is a finite group, with defect group $D=(C_2)^5$ and inertial quotient $E$ . Then one of the following holds: 
\begin{enumerate}[label=(\arabic*)]\itemsep0em
\item $E= \{1\}$ and $k(B)=32$, $l(B)=1$.
\item $E \cong C_3$, $|C_D(E)|=8$ and $k(B)=32$, $l(B)=3$.
\item $E \cong C_3$, $|C_D(E)|=2$ and $k(B)=16$, $l(B)=3$.
\item $E \cong C_5$, and $k(B)=16$, $l(B)=5$.
\item $E \cong C_7$, and $k(B)=32$, $l(B)=7$.
\item $E \cong C_3 \times C_3$, and $k(B)=32$, $l(B)=9$ or $k(B)=16$, $l(B)=1$.
\item $E \cong C_{15}$, and $k(B)=32$, $l(B)=15$.
\item $E \cong C_7 \rtimes C_3$, $|C_D(E)|=4$, and $k(B)=32$, $l(B)=5$.
\item $E \cong C_{21}$ or $C_7 \rtimes C_3$, $|C_D(E)|=1$, and $k(B)=32$, $l(B)=21$ or $k(B)=24$, $l(B)=13$ or $k(B)=16$, $l(B)=5$.
\item $E \cong C_{31}$, and $k(B)=32$, $l(B)=31$ or $k(B)=24$, $l(B)=23$ or $k(B)=16$, $l(B)=15$ or $k(B)=8$, $l(B)=7$.
\item $E \cong (C_7 \rtimes C_3) \times C_3$, and $k(B) \leq 32$ and $k(B)-l(B)=17$ or $k(B)-l(B)=9$.
\item $E \cong C_{31} \rtimes C_5$, $k(B) \leq 32$ and $k(B)-l(B)=5$.
\end{enumerate}
\end{proposition}
\begin{proof}
Since $D$ is abelian, the inertial quotient $E=N_G(D,e)/C_G(D)$ is a subgroup of $\Aut(D)=\GL_5(2)$. First, we give a classification of the subgroups of odd order of $\GL_5(2)$, which gives us all the possibilities for the isomorphism class of $E$ and its action on $D$.

An explicit computation (using Magma \cite{magma}) gives the following diagram. We write $P \to Q$ if there is a subgroup $R \lhd Q$ such that $P \cong R$, $P$ acts on $D$ in the same way as $R$ does, and $|Q|/|R|$ is a prime. We write $P \dashrightarrow Q$ if there is a subgroup $R \leq Q$ as above, except that it is not normal.

{\small 
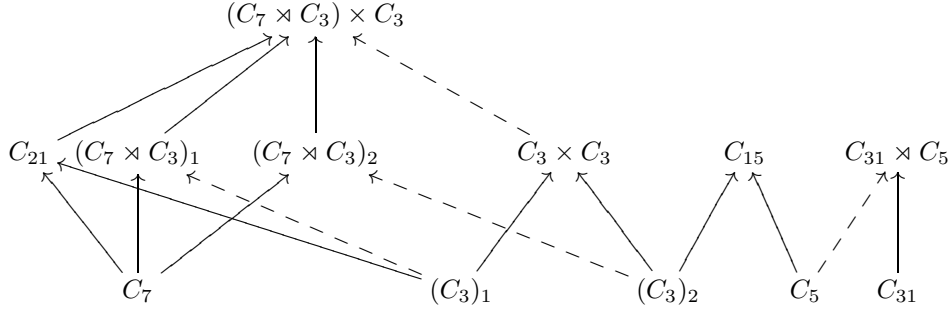
\begin{displaymath}
    \xymatrix@R=36pt@C=3pt{
\quad	 		& \quad				& (C_7 \rtimes C_3) \times C_3	&  				& \quad 				&\quad 						& & \quad			& \quad		\\
 C_{21} \ar[urr]	&(C_7 \rtimes C_3)_1 \ar[ur]	&(C_7 \rtimes C_3)_2 \ar[u]	& 				& C_3 \times C_3 \ar@{-->}[ull] 		&\quad							& C_{15} 	& \quad 			&  C_{31} \rtimes C_5\\
\quad  		& C_7	\ar[u] \ar[ul] \ar[ur]	& \quad 				& (C_3)_1 \ar@{-->}[ull] \ar[ur] \ar[ulll]	& \quad 				& (C_3)_2	 \ar@{-->}[ulll] \ar[ul] \ar[ur]			&  	& C_5	 \ar[ul] \ar@{-->}[ur]	& 	C_{31} \ar[u]	\\
}
\end{displaymath}}

\normalsize \noindent We distinguish the two different actions of $C_3$ and $C_7 \rtimes C_3$ (corresponding to two pairs of distinct conjugacy classes in $\GL_5(2)$) as follows: we denote by $(C_3)_1$ the action on $D$ such that $C_D(C_3)\cong (C_2)^3$, and we denote by $(C_3)_2$ the action such that $C_D(C_3)=C_2$: the generator of this subgroup is the $5$th power of the Singer cycle $C_{15}$ in $\GL_4(2)$. Similarly, we denote by $(C_7 \rtimes C_3)_1$ the action on $D$ such that $C_D(C_7 \rtimes C_3) = (C_2)^2$, or equivalently the one such that the subgroup $C_3 \leq C_7 \rtimes C_3$ acts as $(C_3)_1$. We denote by $(C_7 \rtimes C_3)_2$ the other one, where $C_D(E)=1$ and $C_3 \leq C_7 \rtimes C_3$ acts as $(C_3)_2$. We have shown that $E$ can only be as specified in the statement of the theorem. 

Whenever $C_D(E) \neq 1$, we can use Proposition 16 in \cite{sam17} and immediately obtain our claims. This proves cases (1)-(8).

From Proposition 21 in \cite{sam17}, we have that $k(B) \leq 32$. Now we use the same argument as in \cite[2.1]{kusam13}. A subsection is defined as $(u, b_u)$ where $u \in D$ and $b_u$ is a block of $C_G(u)$ such that $(\langle u \rangle, b_u)$ is a $B$-subpair. Whenever there is a nontrivial subsection $(u,b_u)$ such that $l(b_u)=1$ then $|D|=32$ is a sum of $k(B)$ odd squares of integers, which implies that $k(B) \in \{8, 16, 24, 32\}$. In particular, such a subsection always exists when $E$ is abelian \cite[1.2]{rob92}. Since $D$ is abelian, $B$ is a controlled block, meaning that the fusion system $\mathcal{F}(B) \cong \mathcal{F}_D(D \rtimes E)$. So to compute subsections it is enough to consider a set of representatives $\mathcal{R}$ of the orbits of $D$ under the action of $E$ in the group $D \rtimes E$. 
Recall that from Brauer's second main theorem $k(B) = \sum_{(u,b_u), u \in \mathcal{R}} l(b_u)$, so in particular $k(B)-l(B)=\sum_{(u,b_u), u \in \mathcal{R}, u \neq 1} l(b_u)$. To determine $l(b_u)$ for various subsections, we use cases (1)-(8). We proceed:
\begin{enumerate}
\item[(9)] If $E$ is abelian, then there are four subsections $(1, B)$, $(u_1, b_1)$, $(u_2, b_2)$, $(u_3, b_3)$ with $l(b_1)=3$, $l(b_2)=7$, $l(b_3)=1$. So $k(B)-l(B)=11$ and our claim is proved. \\
If $E$ is not abelian, there are four subsections $(1, B)$, $(u_1, b_1)$, $(u_2, b_2)$, $(u_3, b_3)$ with $l(b_1)=3$, $l(b_2)=7$, $l(b_3)=1$. In particular, there is a subsection of length $1$ so $k(B) \in \{8, 16, 24, 32\}$. Now $k(B)-l(B)=11$ and we are done.
\item[(10)] In this case $E$ is abelian and there are only two subsections, $(1,B)$ and $(u, b_u)$, with $l(b_u)=1$. So $k(B)-l(B)=1$ and we are done.
\item[(11)] In this case there are four subsections $(1, B)$, $(u_1, b_1)$, $(u_2, b_2)$, $(u_3, b_3)$ with   $l(b_1)=9$ or $l(b_1)=1$ (we are in case (6) here), $l(b_2)=5$, $l(b_3)=3$, hence $k(B)-l(B)=17$ or $k(B)-l(B)=9$ and we are done.
\item[(12)] In this case there are only two subsections, $(1,B)$ and $(u,b_u)$, with $l(b_u)=5$. So $k(B)-l(B)=5$ and we are done. \qedhere
\end{enumerate} 
\end{proof}

The proof of our main theorem is based on studying blocks of chains of normal subgroups, and as the starting case we have blocks of quasisimple groups with an (elementary) abelian defect group, which have been completely classified in \cite{ekks}. For the reader's convenience, we extract the quasisimple groups relevant for our case from the main result.

\begin{proposition}[\cite{ekks}] \label{quasisimpleblocks}
Let $G$ be a quasisimple group, and let $B$ be a block of $\cO G$ with defect group $D \neq \{1\}$ contained in $(C_2)^5$. Then at least one of the following occurs: \begin{enumerate}[label=(\arabic*)] \setlength\itemsep{0em}
\item $B$ is the principal block, $G$ is simple and $G \cong \SL_2(8)$, $\SL_2(16)$, $\SL_2(32)$, $J_1$ or ${}^2G_2(q)$ with $q=3^{2m+1}, m \geq 1$. 
\item $B$ is the unique nonprincipal block of $G\cong \operatorname{Co}_3$ with defect group $(C_2)^3$.
\item $B$ is Morita equivalent to a block $C$ with an isomorphic defect group $D$ of $\cO M$ where $M= M_0 \times M_1$ is a subgroup of $G$ such that $M_0$ is abelian and the block of $M_1$ covered by $C$ has defect group $C_2 \times C_2$. In this case, $G$ is of type $D_n(q)$ or $E_7(q)$.
\item $O_2(Z(G))\leq (C_2)^3$ and $D/O_2(Z(G))$ has defect group $C_2 \times C_2$.
\item $B$ is nilpotent covered. In this case, if $B$ is not nilpotent then $G/Z(G)$ is of type $A_n(q)$ or $E_6(q)$ where $q$ is a power of an odd prime.
\end{enumerate}
\end{proposition}
\begin{proof}With the exception of the second claim of (5), the result follows from Theorem 6.1 and Proposition 5.3 of \cite{ekks} (see also 4.1 in \cite{eali18c}) and the fact that all the Schur multipliers of the groups listed in (1) and (2) are trivial, and so in those cases $G$ is actually simple. The second claim of (5) is implied by Lemma 4.2 and Proposition 5.4 in \cite{ekks}, since in the proof of Theorem 6.1 the only case in which $B$ is nilpotent covered but not nilpotent is when the hypotheses of Proposition 5.4 are satisfied.\end{proof}

\section{$(G,B)$-local systems}
When classifying Morita equivalence classes of blocks over $k$, the main theorem of \cite{kk96} is a crucial result. We state it for the reader's benefit:
\begin{theorem} \label{koshkul3}
Let $N \lhd G$ such that $[G:N]$ is a power of $p$. Let $b$ be a block of $k N$ and let $B$ be the unique block of $kG$ that covers $b$. If $B$ has an abelian defect group $D$ that admits a decomposition $D= (D \cap N) \times Q$, then $B \cong b \otimes kQ$. In particular, $B$ is Morita equivalent to the block $b \otimes kQ$ of $k (N \times Q)$. 
\end{theorem}

One of the main obstacles that we face is that this theorem does not immediately generalise to blocks defined over $\cO$. To circumvent this issue, we use a result of Watanabe. We denote the Brou\'e-Puig construction \cite{brpu80} by $*$:
\begin{lemma}[\cite{wa00}] \label{wat00}
Let $N \lhd G$ such that $[G:N]$ is a power of $p$. Let $b$ be a block of $\cO N$ that is covered by a block $B$ of $\cO G$ and also $G$-stable. Suppose that $D$, a defect group of $B$, decomposes as $D=Q\times (D \cap N)$, and let $C=C_G(Q)$. Let $e$ be a root of $B$ in $\cO C_G(D)$, and let $e_Q:=e^{C}$. Suppose that there exists a perfect isometry $I$ between $\mathcal{L}_K(C, e_Q)$ and $\mathcal{L}_K(G, B)$, with the following property:
$$I(\lambda * \chi) = \lambda * I(\chi) \quad (\forall \lambda \in \operatorname{Irr}(Q), \; \forall \chi \in \operatorname{Irr}(e_Q)).$$
Then $B \cong \cO Q \otimes_{\cO} b \text{\quad (as $\cO$-algebras)}$.
\end{lemma}

To ascertain the existence of this perfect isometry and replicate Theorem \ref{koshkul3} for blocks defined over $\cO$, we use the theory of $(G,B)$-local systems, developed in \cite{puigusami93}  (see also \cite{usami96}).

Let $B$ be a block of $\cO G$ with defect group $D$ and inertial quotient $E$. Let $b$ be its Brauer correspondent in $kN_G(D)$: by the main result of \cite{kul85}, $b$ is Morita equivalent to a twisted group algebra of $D \rtimes E$. Following \cite[1.2]{usami96}, we can construct a specific $k^\times$-central extension $\hat{E}$ of $E$ (defined explicitly in \cite[2.4]{puigusami93}), put $\hat{L}:=D \rtimes \hat{E}$ and get a twisted group algebra $k_* \hat{L}$ that is in the same Morita equivalence class of $b$. By \cite[1.2]{usami96}, there are a finite subgroup $L' \leq \hat{L}$ and a block idempotent $e_{b'}$ of $kL'$ such that $k_* \hat{L} \cong k L'e_{b'}$, so we can consider $k_*\hat{L}$ as a block algebra. This twisted group algebra can also be defined over $\cO$ (and hence over $\K$) \cite[5.12]{pu88}.

Some notation: $\mathcal{L}_K(\hat{L})$ denotes the generalised character group of $K_* \hat{L}$ (note the asterisk: every time subgroups of $\hat{L}$ appear, we are considering characters in the twisted group algebra, considered as a block). $\mathcal{L}^0_\K (\hat{L})$ denotes the kernel of the restriction map $d_{\hat{L}}: \mathcal{L}_K (\hat{L}) \to \mathcal{L}_k (\hat{L})$. We use analogous definitions for $KGB$, writing $(G,B)$ instead of $KGB$.

The precise definition of a $(G,B)$-local system \cite{puigusami93} is very technical, and we omit it for brevity. We are interested in the fact that, given an upwardly closed $E$-stable set $X$ of subgroups of $D$, the existence of a $(G,B)$-local system implies the existence of a collection of perfect isometries for each $Y \in X$:
$$\Delta_Y: \mathcal{L}_{\K}(C_{\hat{L}}(Y)) \to \mathcal{L}_\K(C_G(Y),e^{C_G(Y)})$$
that satisfy $\Delta_Y(\lambda * \eta) = \lambda * \Delta_Y(\eta)$ for any $\lambda \in \mathcal{L}_\K (D)^{C_E(Y)}$, $\eta \in \mathcal{L}_\K (C_{\hat{L}}(Y))$. Note that in particular, if $\{1\} \in X$ (and hence $X$ contains all subgroups of $D$), $\Delta_{\{1\}}$ is a perfect isometry between $\cO_* \hat{L}$ and $\cO GB$ that respects the Brou\'e-Puig construction. Then we can apply the same argument to $\cO C_G(Q) e_Q$ in the context of Lemma \ref{wat00}, which determines the same twisted group algebra $\cO_* \hat{L}$. Composing the isometries allows to satisfy the hypothesis of Lemma \ref{wat00}.

In \cite{puigusami93}, Puig and Usami  define an inductive process that enables us to verify the existence of a $(G,B)$-local system defined over the set of all subgroups of $D$. We describe it briefly. Let $X$ be an upwardly closed $E$-stable set of subgroups of $D$ on which a $(G,B)$-local system exists, and consider a maximal subgroup $Q$ that is not contained in $X$. Note that when $X=\{ D \}$ there is always a $(G,B)$-local system by \cite[3.4.2]{puigusami93}. For any $H$ such that $Q \lhd H$, let $\overline{H}$ denote the quotient $H/Q$, and for a block $e$ of $H$ let $\overline{e}$ denote the corresponding block of $\overline{H}$. 

From \cite[3.7]{puigusami93}, there is a bijective isometry induced by the $(G,B)$-local system on $X$: $$\overline{\Delta}^0_Q : \mathcal{L}^0_K (\overline{C_{\hat{L}}(Q)}) \to \mathcal{L}^0_K (\overline{C_G(Q)},  \overline{e}^{C_G(Q)})$$

We have the following lemma:

\begin{lemma}[3.11, \cite{puigusami93}] \label{extend}
With the notations above, a $(G,B)$-local system on $X$ can be extended to a $(G,B)$-local system on $X \cup \{Q^e\}_{e \in E}$ if and only if $\overline{\Delta}^0_Q$ can be extended to an $N_E(Q)$-stable bijective isometry: $$\overline{\Delta}_Q : \mathcal{L}_K (\overline{C_{\hat{L}}(Q)}) \longrightarrow \mathcal{L}_K (\overline{C_G(Q)}, \overline{e}^{C_G(Q)})$$ 
\end{lemma}

To prove Theorem \ref{maintheorem}, we need to show the existence of $(G,B)$-local systems for blocks of $\cO G$ with defect group $D=(C_2)^5$ and inertial quotient whose action on $D$ has at least a fixed point. From Proposition \ref{sambalecalcs}, this means that we are interested in blocks with inertial quotient that is cyclic, or $C_3 \times C_3$, or $(C_7 \rtimes C_3)_1$. The cyclic case has been done in \cite{wa05}. The PhD thesis from which this paper is extracted contains detailed computations of each case, but here we include a shorter proof that uses GAP \cite{gap}, a method originally described in \cite{arsa19}.

\begin{proposition} \label{gblocalnoncyclic}
Let $G$ be a finite group and $B$ be a block of $\cO G$ with defect group $D \cong (C_2)^5$ and inertial quotient $E \in\{C_3 \times C_3, (C_7 \rtimes C_3)_1\}$ (as denoted in Proposition \ref{sambalecalcs}). Then there is a $(G,B)$-local system on the set of all subgroups of $D$.
\end{proposition}
\begin{proof}
The strategy is to proceed with an inductive argument on $X$, an upwardly closed $E$-stable set of subgroups of $D$. As a base case, when $X=\{ D \}$, a $(G,B)$-local system on $X$ exists by \cite[3.4.2]{puigusami93}. Now suppose that there is a $(G,B)$-local system on $X$, and let $Q$ be a subgroup of $D$ maximal with respect to the property $Q \not \in X$. We consider $\tilde{X} = X \cup \{Q^x\}_{x \in E}$, and prove that the isometry $\overline{\Delta}^0_Q$ can always be extended to an $N_E(Q)$-stable isometry. Then Lemma \ref{extend} proves the result. 

Keeping the notation of Lemma \ref{extend}, put $\overline{C_Q}:= \overline{C_{\hat{L}}(Q)}$ and $\overline{e_Q}:= \overline{e}^{C_G(Q)}$. We can automize the computation of all extensions of a given $\overline{\Delta}^0_Q$, as follows: let $\Irr(\overline{C_Q}) = \{\chi_j\}_{j=1}^{k(\overline{C_Q})}$. We fix a basis $\{\theta_i\}_{i=1}^{k(\overline{C_Q})-l(\overline{C_Q})}$ of $\mathcal{L}_K^0(\overline{C_Q})$, and define the $k(\overline{C_Q}) \times(k(\overline{C_Q})-l(\overline{C_Q}))$-matrix $M^0:= (m_{ij})$, where $\theta_i = \sum_{j=1}^{k(\overline{e_Q})} a_{ji} \chi_j$. Then $N^0:= M^t M = (\theta_i, \theta_j) = (\overline{\Delta}^0_Q(\theta_i), \overline{\Delta}^0_Q(\theta_j))$. Finding an extension of $\overline{\Delta}^0_Q$ is equivalent to finding a solution of the equation $Y^t Y = N^0$ with $k(\overline{e_Q}$ rows, as each fixed solution determines equations $\overline{\Delta}^0_Q(\theta_i) = \sum_{j=1}^{k(\overline{e_Q})} y_{ji} \psi_j$, where $\Irr(\overline{e_Q}) = \{\psi_j\}_{j=1}^{k(\overline{e_Q})}$, and hence allows to define the extension $\overline{\Delta}(\chi_j):=\psi_j$. Any time the solution is unique up to permutations and signs of rows, $N_E(Q)$-stability will be automatically guaranteed. All solutions can be computed by running the command $\texttt{OrthogonalEmbeddings}$ on $N^0$ in GAP \cite{gap}.

We summarize: we need to apply Lemma \ref{extend} at each induction step, so we need to build an $N_E(Q)$-stable extension of a given isometry $\Delta_Q^0$. For each $Q$ in which $C_E(Q)$ is cyclic we use the main theorem of \cite{wa05} to say that $\Delta^0_Q$ can be extended. Otherwise, we choose a basis of $\mathcal{L}^0_K (C_{\hat{L}}(Q))$, compute the matrix $N^0$ and run the command $\texttt{OrthogonalEmbeddings}$ on $N^0$, and determine an extension of $\Delta^0_Q$. Then, in each case we need to show that this extension is $N_E(Q)$-stable.

\sloppy Suppose that $E=C_3 \times C_3$. From Proposition \ref{sambalecalcs}, either $k(B)=32$, $l(B)=9$ or $k(B)=16$, $l(B)=1$. Let $D=[D,E] \times C_D(E) = P \times R$ where $R=C_D(E)=C_2$. From Theorem 1 in \cite{wa91}, $k(B)=k(e^{C_G(R)})$, and from Theorem 1.22 in \cite{sam14} $k(e^{C_G(R)})=2k(\overline{e^{C_G(R)}})$. Since $\hat{L}$ is determined locally and $N_G(D,e) \leq C_G(R)$, then $e^{C_G(R)}$ determines the same group $\hat{L}$ as $B$. Then, from the classification in \cite{ea17}, the central extension $\hat{L}/R$ is split if and only if $k(\overline{e^{C_G(R)}})=16$. Otherwise, $k(\overline{e^{C_G(R)}})=8$ and $\hat{L}/R$ is not split. Since $R$ is a direct factor of $\hat{L}$, this implies that $k(B)=32$ if and only if $\hat{L}$ splits, and $k(B)=16$ otherwise. 

\begin{itemize}
\item[($A$)] First we investigate the situation with $k(B)=32$. In this case, $\cO_* \hat{L} \cong \cO L$ (see \cite[10.4]{thevenaz}) where $L=D \rtimes E \cong (A_4 \times A_4) \times C_2$. Then:

\item Suppose that $C_E(Q)=1$. Then $\overline{e_Q}$ is nilpotent and there is a unique extension $\overline{\Delta}_Q$, defined in \cite[4.4]{puigusami93}. 
\item Suppose that $C_E(Q)=C_3$. Then $\overline{e_Q}$ has cyclic inertial quotient, so by the main theorem of \cite{wa05} there is an extension $\overline{\Delta}_Q$. if $Q=C_2$, $N_E(Q) = C_3$ and hence $N_E(Q)$-stability is automatic. Otherwise, $N_E(Q)=E$, and $N_E(Q)/C_E(Q)$ acts trivially on $\Irr(\overline{C_L(Q)})$. We want to show that $E$ also fixes every element in $\Irr(\overline{C_G(Q)}, \overline{e_Q})$; equivalently, we need to show that $E$ acts trivially on $\Irr_k(\overline{C_G(Q)},\overline{e_Q})$. We use an argument based on \cite[4.9]{puigusami93}:

Suppose that $E$ does not act trivially: For any $\phi \in \Irr_k(\overline{C_G(Q)}, \overline{e_Q})$ the induced character $\operatorname{Ind}^{\overline{N_G(Q,e_Q)}}_{\overline{C_G(Q)}}( \phi)$ is an irreducible Brauer character. Since $\overline{N_G(Q,e_Q)}/\overline{C_G(Q)}\cong E/C_E(Q)=C_3$, there is just one isomorphism class of simple $k(\overline{N_G(Q,e_Q)})\overline{e_Q}$-modules. 

Now suppose that $E$ acts trivially. Since $E/C_E(Q)$ is cyclic, then each $\phi \in \Irr_k(\overline{C_G(Q)}, \overline{e_Q})$ can be extended to a Brauer character of $\overline{N_G(Q,e_Q)}$, and the induced character is the sum of three irreducible Brauer characters: hence, there are exactly nine isomorphism classes of simple $k(\overline{N_G(Q,e_Q)})\overline{e_Q}$-modules. So the number of simple $k(\overline{N_G(Q,e_Q)})\overline{e_Q}$-modules is either $1$ or $9$, and this is determined by the action of $E/C_E(Q)$ on $\Irr_k(\overline{C_G(Q)}, \overline{e_Q})$.

From Lemma 3.14 in \cite{puigusami93}, there is a bijection that preserves defect groups and inertial quotients between blocks of $k(\overline{N_L(Q)})$ and blocks of $k(\overline{N_G(Q,e_Q)})$ that cover $\overline{e_Q}$. In particular, in our case $k(\overline{N_L(Q)})$ always has three blocks with three simple modules each, so in our situation $k(\overline{N_G(Q,e_Q)})\overline{e_Q}$ also has nine simple modules: therefore, $E/C_E(Q)$ acts trivially on $\Irr_k(\overline{C_G(Q)},\overline{e_Q})$. In particular, $\overline{\Delta}_Q$ is $N_E(Q)$-stable.
\item Suppose that $C_E(Q)=E$. Then either $Q=1$ (in which case $\overline{e_Q}=e$) or $Q=R=C_2$. From Theorem 1 in \cite{wa91}, $k(e)=32$, $l(e)=9$, and therefore, when $Q=R$, $k(\overline{e_Q})=16$ and $l(\overline{e_Q})=9$. The GAP program shows that in each case there is a unique solution with $32$ (when $Q=1$) or $16$ (when $Q=R$) irreducible characters, which defines an extension $\overline{\Delta}_Q$. Since $N_E(Q)=C_E(Q)$, it is automatically $N_E(Q)$-stable.

\item[($B$)] Suppose that $k(B)=16$. Then the central extension $\hat{E}$ does not split, and we have $\cO_* \hat{L} \cong \cO L' b'$, where $L'=((C_2)^4 \rtimes 3^{1+2}_{\pm}) \times C_2$, and $b'$ is a nonprincipal block of $\cO L'$. Note that it does not matter whether we pick the central extension $3^{1+2}_+$ or $3^{1+2}_-$ nor which nonprincipal block we choose, because all these blocks are all Morita equivalent by the classification in \cite{ea17}.
\item Suppose that $C_E(Q)=1$. Then $\overline{e_Q}$ is nilpotent and there is a unique extension $\overline{\Delta}_Q$, defined in \cite[4.4]{puigusami93}. 
\item Suppose that $C_E(Q)=C_3$. Then $\overline{e_Q}$ has cyclic inertial quotient, so by the main theorem of \cite{wa05} there is an extension $\overline{\Delta}_Q$. As in case (A), we need to show $N_E(Q)$-stability whenever $N_E(Q)=E$. We use an identical strategy, except that in this case we have to consider the situation of $E/C_E(Q)$ \emph{not} acting trivially on $\Irr_k(\overline{e_Q})$, which determines the number of simple $k(\overline{N_G(Q,e_Q)})\overline{e_Q}$-modules to be $1$. We can always choose the extension $\overline{\Delta}_Q$ in a way that agrees with the action of $E$ on $\Irr_k(\cO_* \overline{C_{\hat{L}}(Q)})$ and $\Irr_k(\overline{e_Q})$, which implies $N_E(Q)$-stability by Lemma 3.14 in \cite{puigusami93}. 
\item Suppose that $C_E(Q)=E$. Then either $Q=1$ (in which case $\overline{e_Q}=e$) or $Q=R=C_2$. From Theorem 1 in \cite{wa91}, $k(e)=16$, $l(e)=8$, and therefore, when $Q=R$, $k(\overline{e_Q})=8$ and $l(\overline{e_Q})=1$. GAP shows that in each case there is a unique solution with $16$ (when $Q=1$) or $8$ (when $Q=R$) irreducible characters, which defines an extension $\overline{\Delta}_Q$. Since $N_E(Q)=C_E(Q)$, it is automatically $N_E(Q)$-stable.

\end{itemize}

Suppose that $E=(C_7 \rtimes C_3)_1$. From Proposition \ref{sambalecalcs}, $k(B)=32$ and $l(B)=5$. We have $\cO_* \hat{L} \cong \cO L$ (again from \cite[10.4]{thevenaz}) where $L=D \rtimes E$. Let $D=[D,E] \times C_D(E) = P \times R$ where $R=C_D(E)=(C_2)^2$.
Note that, while a priori $C_E(Q) \in \{1, C_3, C_7, C_7 \rtimes C_3 \}$, if $C_7 \subseteq C_E(Q)$ then $C_E(Q)=C_7 \rtimes C_3$. So we only have to consider three possibilities for $C_E(Q)$.

\begin{itemize}
\item Suppose that $C_E(Q)=1$. Then $\overline{e_Q}$ is nilpotent and there is a unique extension $\overline{\Delta}_Q$, defined in \cite[4.4]{puigusami93}. 
\item Suppose that $C_E(Q)=C_3$. Then $\overline{e_Q}$ has cyclic inertial quotient, so by the main theorem of \cite{wa05} there is an extension $\overline{\Delta}_Q$. In this case note that $Q \cap P = C_2$, so since $C_7 \lhd E$ does not centralise $Q \cap P$ then $N_E(Q)=C_E(Q)$, so $\overline{\Delta}_Q$ is automatically $N_E(Q)$-stable.
\item Suppose that $C_E(Q)=E$. Then $Q \in \{1, C_2, (C_2)^2\}$. From Proposition \ref{sambalecalcs}, $k(e)=32$. In each case GAP shows that there is a unique solution with the appropriate number of irreducible characters, which defines an extension that is automatically $N_E(Q)$-stable since $N_E(Q)=C_E(Q)$. \qedhere
\end{itemize}
\end{proof}

\begin{proposition} \label{index2}
Let $G$ be a finite group and let $B$ be a block of $\cO G$ with defect group $D \cong (C_2)^5$ and inertial quotient $E$ such that $|C_D(E)|\neq 1$. Suppose that there is $N \lhd G$ with $[G:N]=2^i$ and let $b$ be the unique block of $\cO N$ covered by $B$. Let $Q \leq D$ with $G=NQ$. Then $B$ is Morita equivalent to the block $b \otimes \cO Q$ of $\cO (N \times Q)$. 
\end{proposition}
\begin{proof}
\sloppy Whenever $E$ is cyclic of order $3$, $5$, $7$ or $15$, Proposition \ref{sambalecalcs} implies that $l(B)=|E|$. Then the main theorem of \cite{wa05} shows that there is a $(G,B)$-local system on the set $X$ of all subgroups of $D$. When $E=C_3 \times C_3$ or $(C_7 \rtimes C_3)_1$, there is a $(G,B)$ local system on $X$ as shown in Proposition \ref{gblocalnoncyclic}.

In particular in each case we have a perfect isometry $$\Delta_1 : \mathcal{L}_\K(\hat{L}) \longrightarrow \mathcal{L}_\K (G, B)$$ that respects the $*$-construction.

We consider $C_G(Q)$ and $e^{C_G(Q)}$: this block also has defect group $D$ and inertial quotient $E$, and since $N_G(D,e) \leq C_G(Q)$, and hence $N_{C_G(Q)}(D,e) = N_G(D,e)$, it determines the same twisted group algebra $\cO_* \hat{L}$ as $B$. Again by the results above, there is a $(C_G(Q),e^{C_G(Q)})$-local system, which gives perfect isometry 
$$\Delta^Q_1: \mathcal{L}_\K (C_{\hat{L}}(Q))  \longrightarrow \mathcal{L}_\K (C_G(Q), e^{C_G(Q)})$$
that respects the $*$-construction.

Now $P= \Delta_1 \circ (\Delta_1^Q)^{-1}$ is a perfect isometry that respects the $*$-construction, so we can apply Lemma \ref{wat00}. We are done.
\end{proof}

%

\section{Crossed products and Picard groups}
We recall the key concepts from \cite{kul95}. Given a finite group $G$ and a ring with identity $A$, $A$ is a $G$-graded ring if there is a decomposition $A=\bigoplus_{g \in G} A_g$ as additive subgroups such that $A_g A_h \subseteq A_{gh}$, and $A_1$ is a subring of $A$ containing $1$. \\
A $G$-graded ring $A$ is said to be a \emph{crossed product} of $A_1$ with $G$ if for any $g \in G$, $A_g$ contains at least one unit. We call two $G$-graded rings $A$ and $B$ \emph{weakly equivalent} if there is an isomorphism of rings $\phi: A \to B$ such that $\phi(A_g) \subseteq B_g$ for all $g \in G$. Moreover, we say they are \emph{equivalent} if $\phi$ restricts to the identity map on $A_1 \cong B_1$.

A key result from K\"ulshammer's article is a characterization of all the possible crossed products of a given ring $R$ and a group $G$:
\begin{theorem} \label{crossedproductseqclass}
\sloppy The equivalence classes of crossed products of a ring $R$ with a group $G$ are parametrized by pairs $(\omega, \zeta)$, where $\omega: G \to \Out(R)$ is a homomorphism whose corresponding $3$-cocycle in the sense of \cite{kul95} in $H^3(G, \mathcal{U}(Z(R)))$ vanishes, and $\zeta \in H^2 (G, \mathcal{U}(Z(R)))$ where the action of $G$ on $\mathcal{U}(Z(R))$ is induced by $\omega$.
Moreover, weak equivalence classes of crossed products correspond to orbits of $\Aut(R)$ on the set of possible $(\omega, \zeta)$.
\end{theorem}

Given a Morita equivalence class we consider a canonical representative of it: a basic algebra. It is well known that two Morita equivalent algebras have isomorphic basic algebras, and that any algebra is Morita equivalent to its basic algebra. This is compatible with the crossed product structure, as the following lemma shows.

\begin{lemma} \label{crossact}
Let $G$ be a finite group, $N \lhd G$ with $[G:N]$ a prime $\ell \neq p$. Let $X=G/N$. Let $B$ a block of $\cO G$ that covers a $G$-stable block $b$ of $\cO N$, and let $f$ be a basic idempotent of $b$, i.e. an idempotent such that $fbf$ is a basic algebra of $b$. Then $fBf$ is a crossed product of $fbf$ with $X$, and $fBf$ is Morita equivalent to $B$.
\end{lemma}
\begin{proof}
The group algebra $\cO G$ is a crossed product of $\cO N$ and $X=G/N$. Since $b$ is $G$-stable, $B=\cO G B$ is also a crossed product of $b=\cO N b$ with $X$. The first claim now follows from Proposition 4.15 in \cite{eisele18}, as for two basic idempotents $e$ and $f$ of $A$ the property $eA \cong fA$ holds since basic idempotents are in the same orbit under conjugation by units of $A$. 

Recall that for an algebra $A$ and an idempotent $f$, $A$ and $fAf$ are Morita equivalent if and only if $AfA = A$ \cite[9.9]{thevenaz}. Since $b$ and $fbf$ are Morita equivalent, $bfb = b$, hence:
$$BfB = BbfbB = BbB = B$$\vspace*{-1em}\qedhere
\end{proof}

To apply Theorem \ref{crossedproductseqclass}, in principle we need to determine $\Out(fbf)$ for each block $b$. Instead, we adopt a different point of view. We give the relevant definitions:

Recall that two algebras $A$ and $B$ are Morita equivalent if and only if there is an $A$-$B$-bimodule $M$ and a $B$-$A$-bimodule $N$ such that $M \otimes_B N \cong A$ and $N \otimes_A M \cong B$. The Picard group of a block $b$, denoted by $\Pic(b)$, is the group of $b$-$b$-bimodules that induce a self-Morita equivalence of $b$, where the group operation is given by the tensor product. We are interested in the subgroups $\cT(b) \leq \cE(b) \leq \Pic(b)$, where $\cT(b)$ is defined as the subgroup of all the bimodules in $\Pic(b)$ with trivial source, and $\cE(b)$ as the subgroup of all bimodules with endopermutation source. By \cite{ei19}, $\Pic(b)$ is always a finite group. The following result, extracted from the main theorem of \cite{bkl17}, gives an upper bound for the size of these subgroups.

\begin{lemma}[\cite{bkl17}] \label{picard}
Let $G$ be a finite group, and let $b$ be a block of $\cO G$ with abelian defect group $D$ and inertial quotient $E$. Let $\mathcal{F}$ be the fusion system on $D$ determined by $b$. Let $\operatorname{Pic}(b)$, $\cE(b)$ and $\cT(b)$ be as defined above. Then there are exact sequences:
\begin{equation*}\begin{split}
&1 \longrightarrow \Out_D(A) \longrightarrow \cE(b) \longrightarrow D_\cO (D, \mathcal{F}) \rtimes \Out(D,\mathcal{F})\\
&1 \longrightarrow \Out_D(A) \longrightarrow \cT(b) \longrightarrow \Out(D,\mathcal{F})
\end{split}\end{equation*}
where $A=i \cO G i$ is a source algebra of $b$, $\Aut_D(A)$ is the group of algebra automorphisms of $A$ which fix the image of $D$ in $A$ elementwise, $\Out_D(A)$ is the quotient of $\Aut_D(A)$ by the subgroup of inner automorphisms induced by elements in $(A^D)^\times$, $D_\cO (D, \mathcal{F})$ is the subgroup of $\mathcal{F}$-stable modules of the Dade group of $D$ over $\cO$ and $\Out(D,\mathcal{F}):= \Aut(D,\mathcal{F})/\Aut_\mathcal{F}(D)$. Moreover, $\Out_D(A)$ is isomorphic to a subgroup of $\Hom(E,k^\times)$.
\end{lemma}

By definition, for two Morita equivalent blocks $b$ and $c$, $\Pic(b) \cong \Pic(c)$. When we have stronger equivalences, we can say more. By definition, if two blocks $b$ and $c$ are source algebra equivalent, as defined in Section $2$, then $\cT(b) \cong \cT(c)$. Similarly, a basic Morita equivalence between $b$ and $c$ implies that $\cE(b) \cong \cE(c)$.

In the proof of our result we use the following information about Picard groups.

\begin{proposition}\label{taugroups}
For a block $b$ of $\cO G$ with defect group $D$, the following holds:
\begin{itemize}\itemsep0em
\item[(1)] If $b = \cO(C_2^3 \rtimes C_7)$, then $\Pic(b)=\cT(b)=C_7 \rtimes C_3$.
\item[(2)] If $b=\cO(C_2^3 \rtimes (C_7 \rtimes C_3))$, then $\Pic(b)=\cT(b)= C_3$.
\item[(3)] If $b =\cO(A_4 \times (C_2^3 \rtimes C_7 ))$, then $\Pic(b)=\cT(b)$, and any maximal odd order subgroup is isomorphic to $C_3 \times (C_7 \rtimes C_3)$.
\item[(4)] If $b=\cO(A_4 \times (C_2^3 \rtimes (C_7 \rtimes C_3)))$, then $\Pic(b)=\cT(b)$, and any maximal odd order subgroup is isomorphic to $C_3 \times C_3$.
\item[(5)] If $b=\cO(A_4 \times A_4 \times C_2)$, then $\Pic(b)=\cT(b)$, and any maximal odd order subgroup is isomorphic to $C_3 \times C_3$.
\item[(6)] The unique maximal odd order subgroup of $\cE(B_0(A_4 \times A_5 \times C_2))$ is isomorphic to $C_3$.
\end{itemize}
Moreover, let $Q$ be a finite abelian $2$-group. Then:
\begin{itemize}\itemsep0em
\item[(7)] $\Pic(\cO(A_4 \times Q))= S_3 \times (Q \rtimes \Aut(Q))$.
\item[(8)] $\Pic(B_0(\cO(A_5 \times Q)))= C_2 \times (Q \rtimes \Aut(Q))$.
\end{itemize}
\end{proposition}
\begin{proof}
Cases (1) and (2) follow from Lemma 5.2 and Theorem 4.6 in \cite{eali18} respectively. From the main theorem on \cite{liv19}, in cases (3) and (5) $\Pic(b)=\cT(b)$. Moreover in case (4) $\Pic(b)=\cT(b)$, using Corollary 4.5 in \cite{eali18} and direct inspection of the character table of $A_4 \times (C_2^3 \rtimes (C_7 \rtimes C_3))$. Now we use a technique from the proof of Theorem 4.7 of \cite{eali18}, which in turn uses the main theorem of \cite{bkl17}, to compute these trivial source Picard groups using the exact sequence from Proposition \ref{picard}.

We compute $\Out_D(A)$ using Lemma 2.3 in \cite{eali18} and the information in \cite{eali18} on the outer automorphism groups of source algebras. Further, by Lemma 2.1 in \cite{eali18} we have an identification $\Out(D,\mathcal{F}) \cong N_{\Aut(D)}(E)/E$. Then: \begin{itemize}
\item[(3)]We have that $\Out_D(A) = C_3 \times C_7$ and $\Out(D,\mathcal{F})=C_2 \times C_3$. Therefore, by case (e) of Theorem 4.7 in \cite{eali18} we have a semidirect product of $\Out_D(A)$ and $\Out(D,\mathcal{F})$ and we are done.
\item[(4)]We have that $\Out_D(A) = C_3 \times C_3$, and $\Out(D,\mathcal{F})=(C_2)^3$. Hence, we are done.
\item[(5)]We have that $\Out_D(A) = C_3 \times C_3$ and $\Out(D,\mathcal{F})=C_2 \wr C_2$. Hence, we are done.
\item[(6)]From \cite{dade} (see also \cite[5.2]{geteth}), any finite subgroup of $D_\cO (D, \mathcal{F}) \leq D_\cO (D)$ is a $2$-group, so by Proposition \ref{picard} any subgroup of odd order of $\cE(b)$ is contained in $\cT(b)$. We compute $\Out_D(A) \cong S_3 \times C_2$, and $\Out(D,\mathcal{F}) = (C_2)^3$, so we are done.
\end{itemize}
Cases (7) and (8) are immediate from Theorem 4.6 in \cite{eali18}.
\end{proof}

\begin{method} \label{clubsuit}
\normalfont We detail our method and the context in which it applies: let $G$ be a finite group and $B$ be a block of $\cO G$ with defect group $D \cong (C_2)^5$. Suppose that there is $N \lhd G$ with $[G:N]$ odd (so $G/N$ is solvable) and that $B$ covers a $G$-stable block $b$ of $\cO N$. Moreover, suppose that $C_G(N) \leq N$, and that $N=\ker(G \to \Out(b))$ where the map is given by $G$ acting by conjugation on $b$. Note that, since $[G:N]$ is odd, $B$ and $b$ share a defect group.

Let $f$ be a basic idempotent for $b$. From Lemma \ref{crossact}, $B$ is Morita equivalent to $fBf$, which is a crossed product of $fbf$ with $G/N$. Let $\omega: G/N \to \Out(fbf)$ be the homomorphism that corresponds to the crossed product weak equivalence class of $fBf$ as in Theorem \ref{crossedproductseqclass}, obtained as the composition of:
\begin{equation*}
G/N \xrightarrow{\quad \alpha\quad} \Out_\star(N) \xrightarrow{\quad\beta\quad} \Out(b) \xrightarrow{\quad\gamma\quad} \Pic(b)=\Pic(fbf)=\Out(fbf)
\end{equation*}
where we define the elements as follows: \begin{itemize}
\item[($\star$)] We define $\Out_\star(N)$ as the subgroup $\{ \Phi \in \Out(N) : \forall \phi \in \Phi , \; \phi(b)=b \}$, where we extend $\phi$ linearly to $\cO N$. Since the block idempotent of $b$ is central in $\cO N$, each inner automorphism of $N$ fixes $b$, so the action of any automorphism in the coset $\Phi$ does not depend on the choice of the representative. 
\item[($\alpha$)] For a coset $x \in G/N$, we fix a representative $g \in x$. Let $\tau_g \in \Aut(N)$ be defined as $n \mapsto {}^gn = gng^{-1}$, and let $\alpha(x)$ to be the coset of $\tau_g$ in $\Out(N)$. If we choose a different representative $h \in x$ then, since $h= {}^mg$ for some $m \in N$, ${}^h n = {}^{mgm^{-1}}n$, so the coset in $\Out(N)$ does not depend on the choice of $g$ and $\alpha$ is well defined. Note that, since $b$ is $G$-stable, in our situation $\alpha(G/N) \leq \Out_\star(N)$.
\item[($\beta$)] For a coset $\Phi \in \Out_\star(N)$, choose a representative $\phi \in x$, and let $\phi \in \Aut(\cO N)$ be the automorphism obtained extending $\phi$ linearly. Since $\phi(b)=b$, we define $\beta(\phi)$ to be the coset in $\Out(b)$ of the restriction of $\phi$ to $\cO N b$. Recall that inner automorphisms of $N$ induce inner automorphisms of $b$: in fact, for $n \in N$, we can consider the decomposition of $\cO N$ into blocks, and hence the element $nb \in \cO N b$. Then $\beta$ is well defined.
\item[($\gamma$)] For $\phi \in \Aut(b)$, we define the $b$-$b$-bimodule $\;_\phi b$ as follows: ${}_\phi b = b$ as sets, and $x\cdot m \cdot y=\phi(x)my$ for $x,m,y \in b$. From \cite[55.11]{cure87}, since inner automorphisms give isomorphic bimodules, the map $\gamma$ defined as $\gamma(\phi)={}_\phi b$ gives an embedding of $\Out(b)$ in $\Pic(b)$. 
\end{itemize}
Note that $\alpha$ and $\gamma$ are always injective maps, and since we assumed that $N=\ker(G \to \Out(b))$ so is $\beta$. Then we can identify $G/N$ with a subgroup of $\Pic(b)$ that has odd order.

For any $g \in G$ we have an induced action $\tau_g \in \Aut(N)$ given by conjugation, and a corresponding $\beta(\tau_g) \in \Out(b)$. Since ${}_{\tau_g}b$ is a direct summand of the permutation $\cO N$-$\cO N$-bimodule $ _{\tau_g}\cO N$, then ${}_{\tau_g} b \in \cT(b)$. Therefore, $\gamma\beta\alpha(G/N) \leq \cT(b)$. Hence, a priori we should examine only the possibilities for $\omega$ that correspond to elements in $\cT(b)$. However, in most cases we will only know $b$ up to Morita equivalence, and therefore $\cT(b)$ will not be well-determined since it is not invariant under Morita equivalences in general (and in fact not even for nilpotent blocks, as seen in \cite[7.2]{bkl17}): when given an arbitrary block $b$ Morita equivalent to a block $c$, a priori we can only say that $\cT(b) \leq \Pic(c)$. However, if the equivalence is basic then we can say that $\cT(b) \leq \cE(c)$, and a source algebra equivalence preserves $\cT(b)$.

The solvability of $G/N$ allows us to consider a chain of subnormal subgroups:
\begin{equation*} 
N = N_0 \lhd N_1 \lhd \dots \lhd N_t = G 
\end{equation*}
such that $\ell_i:=[N_{i+1}:N_i]$ is prime, and a block chain $b_i$ of $\cO N_i$ such that $b_i$ covers $b_{i-1}$, with $b_t:=B$.  Since $G/N$ is odd, they all share a defect group. We assume that each $N_i \lhd G$, as this will hold in all cases that arise in our proof.

From Lemma \ref{crossact}, $b_1$ is Morita equivalent to a crossed product of $b$ with $N_1/N$, and the weak crossed product equivalence class is specified by a pair $(\omega_1, \zeta_1)$ as in Theorem \ref{crossedproductseqclass}. As detailed in \cite[\S 3]{ea17} (using \cite{kul95} and \cite[1.2.10]{lin18}) the group $H^2(W,\mathcal{U}(Z(fbf))) = \{1\}$ whenever $W$ is cyclic, so weak equivalence classes of crossed products of $fbf$ and $G/N$ are classified by just orbits of possible $\omega_1$ whose induced $3$-cocycle in the sense of \cite{kul95} vanishes.

In the following we will actually consider each possibility for $\omega_1$ specified by $\Pic(b)$ without checking the additional requirement of the induced $3$-cocycle vanishing: the existence of examples of blocks of finite groups whose crossed product structure induces $\omega_1$ in each case will imply, post hoc, that the induced $3$-cocycle indeed vanishes.

For each possible $\omega_1$, we find an example of groups that realise it, and list all those classes. We then move to the next group extension, $N_1 \lhd N_2$. In each case that we will examine, $G/N$ is either cyclic of prime order, $C_{21}$, $C_7 \rtimes C_3$, $C_3 \times C_3$ or $C_3 \times (C_7 \rtimes C_3)$: in the last four cases, having classified crossed products of $b$ with $C_7$ or $C_3$ is not enough, and we need to proceed further. Since $G/N_1 \cong (G/N)/(N_1/N)$, we know this group up to isomorphism, but we do not know how this quotient acts on $b_1$, i.e. its embedding in $\Pic(b_1)$. This is not merely a technical issue, as the interplay between the subgroup $G/N_1$ identified with a quotient of a subgroup of odd order of $\Pic(b)$ and its embedding in $\Pic(b_1)$ is a reason that a cocycle in addition to $\omega$ is needed to identify the crossed product in Theorem \ref{crossedproductseqclass}. To examine these cases, we use two different techniques: 
\begin{itemize}
\item When $G/N = C_7 \rtimes C_3$, from \cite[5.5.i]{kar93} the group $H^2(C_7 \rtimes C_3, k^\times) = \{1\}$, so we can still use the argument in \cite[\S 3]{ea17} and consider each crossed product of $b$ with $C_7 \rtimes C_3$ directly just as we did for the cyclic case. 
\item In each situation in which $C_{21}$, $C_7 \rtimes C_3$ or $C_3 \times (C_7 \rtimes C_3)$ occur among the possibilities for $G/N$, we are able to compute the maximal odd order subgroups of $\Pic(b_1)$, and hence to list all possible embeddings. Again we find examples that realise each possibility, and then proceed until the process terminates.
\end{itemize}
\end{method}

Note that a priori if $\Pic(b)=C_\ell$ for a prime $\ell$, then there are $\ell$ different crossed product weak equivalence classes for $b_1$: however, once a group $N_1 \rhd N$ with index $\ell$ and a block $b_1$ that realises any of the nontrivial crossed products is found, it is also a realisation of all the other nontrivial crossed products. In fact, if for a fixed generator $y$ of $N_1/N$ $\omega(y)=x$, then $\langle x \rangle = \Pic(b)$. If we pick any other homomorphism $\omega_a$ with $\omega_a(y):=x^a$ then for some $b$ it holds that $\omega_a(y^b)=x$, and since $N\langle y^b \rangle \cong N\langle y \rangle$ then $N_1$ realises $\omega_a$.

In many cases, we use the groups $\PSL_3(7) \lhd \PGL_3(7)$, since the former has a block Morita equivalent to $\cO A_4$ covered by a nilpotent block of the latter. Further, we use the construction in \cite[4.4]{pu11} to produce examples of blocks covered by a block with a smaller inertial quotient, which consists in taking central extensions by a cyclic group and looking at nonprincipal blocks. For instance, the group $G=(C_2)^3 \rtimes 7^{1+2}_+$ has six nilpotent blocks, and each of them covers blocks with inertial quotient $C_7$ of any maximal normal subgroup of index $7$.
 
In order to prove our main result, we need to examine the possible Morita equivalence classes of blocks that cover specific classes of blocks. Throughout this section, we use the labelling of classes introduced in Theorem \ref{maintheorem}.

\begin{proposition} \label{a4ora5}
Let $G$ be a finite group and $B$ be a quasiprimitive block of $\cO G$ with defect group $D \cong (C_2)^5$. Suppose that there is $N \lhd G$ with $[G:N]$ odd, and that $B$ covers a block $b$ of $\cO N$. Moreover, suppose that $C_G(N) \leq N$ and that $N=\ker(G \to \Out(N))$.
\begin{enumerate}[label=(\Roman*)]\itemsep0em
\item If $b$ is Morita equivalent to (ii)$=\cO(A_4 \times (C_2)^3)$, then $B$ is Morita equivalent to one of (i), (ii), (iv), (vi), (viii), (xiii), (xvii), (xx), (xxiv), (a), (b).
\item If $b$ is Morita equivalent to (iii)$=B_0(\cO(A_5 \times (C_2)^3))$, then $B$ is Morita equivalent to one of (iii), (ix), (xiv), (xxv).
\end{enumerate}
\end{proposition}
\begin{proof}
We apply Method \ref{clubsuit}. $\Pic(b)$ is known in both cases from Proposition \ref{taugroups}.

\begin{enumerate}[label=(\Roman*)]
\item $\Pic(b) \cong S_3 \times ((C_2)^3 \rtimes \GL_3(2))$, and $C_3 \times (C_7 \rtimes C_3)$ is a maximal subgroup of odd order. Note that from \cite{eali18} we also know the action of $\Pic(b)$ on the modules of $b$, which is invariant by Morita equivalence. Let $\sigma : G \to S_3$ be the homomorphism given by the action of $G$ permuting the three simple modules of $b$. Note that either $G=\ker(\sigma)$ or $[G:\ker(\sigma)]=3$.

Consider a chain of normal subgroups $\{N_i\}$ of length $t$ where $N_{t-1}=\ker(\sigma)$, and the corresponding block chain $\{b_i\}$. Note that $t \leq 3$ since $G/N$ is isomorphic to a subgroup of $\Pic(b)$.  

The block $b_1$ is Morita equivalent to a crossed product of the basic algebra of $b$ with $X_1=N_1 /N$: let $\omega_1$ be the homomorphism that specifies it in the sense of Theorem \ref{crossedproductseqclass}. There are four nontrivial possibilities for $(X_1, \omega_1)$, which give the following Morita equivalence classes for $b_1$: 
\begin{enumerate}[label=(\arabic*)]
\item $|X_1|=7$, and $b_1$ is Morita equivalent to (xiii), realised when $N=A_4 \times (C_2)^3$.
\item $|X_1|=3$, and $b_1$ is Morita equivalent to (viii), realised when $N=A_4 \times (C_2)^3$.
\item $|X_1|=3$, and $b_1$ is Morita equivalent to (i), realised when $N=\PSL_3(7) \times (C_2)^3$ and $G=\PGL_3(7) \times (C_2)^3$.
\item $|X_1|=3$, and $b_1$ is Morita equivalent to (a), realised when $G=(C_2^4 \rtimes 3_+^{1+2}) \times C_2$ and $N$ is a maximal subgroup of $G$ with index $3$.
\end{enumerate}
In cases $(3)$ and $(4)$ the simple modules of $b$ are not fixed by the action of $N_1/N$, so $N=\ker(\sigma)$ and hence $N_1=G$.

For the other two cases, we consider $b_2$, $N_2/N_1$ and the corresponding $\omega_2$:
\begin{enumerate}[label=(\arabic*)]
\item Note that $C_7 \lhd G/N$, so $N_1 \lhd G$. Then $G/N_1 = (G/N) / (N_1/N) \leq C_3 \times C_3$. From Proposition \ref{taugroups} the image of $G/N_1$ in $\Pic(b_1)$ is contained in $C_3 \times C_3$, so there are three possible embeddings of $C_3$ in $\Pic(b_1)$, which produce the following Morita equivalence classes for $b_2$:
\begin{itemize}
\item (xxiv), realised when $N_1=A_4 \times ((C_2)^3 \rtimes C_7)$.
\item (xvii), realised when $N_1 = \PSL_3(7)\times ((C_2)^3 \rtimes C_7)$.
\item (b), realised when $G=(C_2)^5 \rtimes (C_7 \rtimes 3^{1+2}_+)$, and $N_1$ is a maximal subgroup of $G$ with index $3$.
\end{itemize}
In the last two cases the simple modules of $b$ are not stabilised by the action of $N_2/N$, so $N_1=\ker(\sigma)$ and hence $N_2=G$ and $b_2=B$.

In the first case note that $N_2 \lhd G$ since $C_7 \rtimes C_3 \lhd G/N$. If $G \neq N_2$, clearly $G/N_2 \leq C_3$, so $G=N_3$. From Proposition \ref{taugroups}, $\Pic(b_2)$ contains a unique maximal subgroup isomorphic to $C_3 \times C_3$, so there are again three possible embeddings of $C_3$ in $\Pic(b_2)$, which determine the following Morita equivalence classes for $B$:
\begin{itemize}
	\item (xiii), realised when $G=(C_2)^3 \rtimes (C_7 \rtimes 3^{1+2}_+)$ and $N_2$ is a maximal subgroup of $G$ with index $3$.
	\item (xvii), realised when $N_2 = \PSL_3(7)\times ((C_2)^3 \rtimes (C_7 \rtimes C_3))$.
	\item (xx), realised when  $G=(((C_2)^3 \rtimes C_7)  \times \PSL_3(7))\rtimes 3^{1+2}_+$.
\end{itemize}
Note that, actually, the first one cannot occur as then $G/N$ stabilises the simple modules of $b$, but if $G=\ker(\sigma)$ then $t \leq 2$. We have exhausted all the possibilities.

\item From Proposition \ref{taugroups} the unique maximal subgroup of odd order of $\Pic(b_1)$ is isomorphic to $C_3 \times C_3$, so we can assume that $N_2/N_1 = C_3$, and there are three possible embeddings of $C_3$ in $\Pic(b_1)$, which determine the following Morita equivalence classes for $B$:
\begin{itemize}
\item For two distinct embeddings, $B$ is Morita equivalent to (ii), realised when $N_1 = \PSL_3(7) \times C_2 \times A_4$. If $G=N_2$, we are done. Otherwise $G/N_2=C_3$ and the possibilities for $B$ are the same as the ones determined in cases (2)-(4).
\item \sloppy $B$ is Morita equivalent to (iv), realised as a crossed product\footnote{The group $PSL_3(7)^2 \times C_2$ does not have a normal subgroup $N$ with a block Morita equivalent to $b$, so this group cannot actually appear in our chain. However, we are looking at all possible crossed products between $b_1$ and $C_3$, and this group provides an example of one class: it possibly not occurring does not hinder our classification purpose. Whenever this happens, we say that the class is \textit{realised as a crossed product}.} when $N_1=\PSL_3(7)^2 \times C_2$. Moreover, $G=N_2$ because the simple modules of $b$ are not stabilised by the action of $N_2/N$.
\end{itemize}
\end{enumerate}
\item From Proposition \ref{taugroups}, $\Pic(b) = C_2 \times ((C_2)^3 \rtimes \GL_3(2))$, which contains $C_7 \rtimes C_3$ as a maximal subgroup of odd order. We consider $B$ as a crossed product of $b$ with $G/N$. For each isomorphism type of $G/N$ there is a unique embedding in $\Pic(b)$, and we have the following possibilities, for the Morita equivalence class of $B$:
\begin{itemize}
\item If $[G:N]=7$, then $B$ is Morita equivalent to (xiv), realised when $N=A_5 \times (C_2)^3$.
\item If $[G:N]=3$, then $B$ is Morita equivalent to (ix), realised when $N=A_5 \times (C_2)^3$.
\item If $[G:N]=21$ then $B$ is Morita equivalent to (xxv), realised when $N= A_5 \times (C_2)^3$.
\end{itemize} \qedhere
\end{enumerate} 
\end{proof}

In our proof, we need to look at blocks covering a block of a central product of two quasisimple groups whose Picard group is, at the moment, unknown. In these situations we use the group structure to reduce to a known subgroup of the Picard group, but in the following specific case we can prove a stronger result using Clifford theory. In this situation $G/N$ is a subgroup of odd order of the outer automorphism group of the central product of up to two quasisimple groups, in which case the supersolvability hypothesis is a consequence of the classification of finite simple groups (see \cite{atlas}).
\begin{proposition} \label{a5a5c2}
Let $G$ be a finite group and $B$ be a quasiprimitive block of $\cO G$ with defect group $D \cong (C_2)^5$. Suppose that there is $N \lhd G$ with $[G:N]$ odd, that $G/N$ is supersolvable, and that $B$ covers a $G$-stable block $b$ of $\cO N$. Suppose that $C_G(N) \leq N$ and $N=\ker(G\to \Out(b))$. If $b$ is Morita equivalent to (x)$=B_0(\cO(A_5 \times A_5 \times C_2))$, then $B$ is source algebra equivalent to $b$.
\end{proposition}
\begin{proof}
Since $G/N$ is supersolvable, we can consider a chain of normal subgroups $N \lhd N_1 \lhd \dots \lhd N_t \lhd G$, with prime indices, and a corresponding block chain $b, b_1, \dots, b_t, B$ where each block covers the ones below it. Note that they all share a defect group, and that each $N_i \lhd G$. Consider the action of $G/N$ on $b$ by conjugation. 

Let $[N_1:N]=\ell$, an odd prime. From Lemma \ref{gbstarlemma}, either $b_1$ is source algebra equivalent to $b$ or $b_1$ is the unique block covering $b$. Suppose the latter: $l(b)=9$, and from the decomposition matrix of $b$ we know that, if we consider the character of each projective cover of the simple modules, there are: one with $32$ irreducible constituents, four with $16$ and four with $8$. Any automorphism of the block preserves the number of irreducible constituents: hence, if $b_1$ is the unique block covering $b$ and $\ell \geq 5$ then $N_1/N$ fixes every simple module, which implies that $l(b_1)=9\ell$. This is a contradiction to Proposition \ref{sambalecalcs} since $l(c) \leq k(c) \leq 32$ for any block $c$ with defect group $D$. If $\ell=3$, since $e(b)=9$, from Lemma 4.11 in \cite{mck19} $e(b_1)=1,3, 9$ or $27$. First, note that $e(b_1)=27$ is a contradiction to Proposition \ref{sambalecalcs}. Now either every simple $b$-module is fixed, so $l(b_1)=27$ (again a contradiction), or there is one orbit of length $3$ and six fixed characters, which gives $l(b_1)=1+3\cdot 6 = 19$ (again a contradiction), or there are two orbits of length $3$ and $3$ fixed characters, which gives $l(b_1)=1+1+3\cdot 3 = 11$ (again a contradiction). Since we know that there is at least one fixed simple $b$-module, there cannot be three orbits of length $3$. 

Therefore, $b_1$ is source algebra equivalent to $b$. Since $B$ is quasiprimitive and $N_1 \lhd G$, $b_1$ is $G$-stable. We can now repeat the argument for any other intermediate block $b_i$ (replacing $b_1$ with $b_{i+1}$, and $b$ with $b_i$) and compose the equivalences to obtain that $B$ is source algebra equivalent to $b$.
\end{proof}

\begin{proposition} \label{twocomponents}
Let $G$ be a finite group and $B$ be a quasiprimitive block of $\cO G$ with defect group $D= (C_2)^5$. Suppose that there are $H_1, H_2 \lhd G$, $H=H_1 \times H_2$ with $H \lhd G$ and $[G:H]$ odd, and that $B$ covers $G$-stable blocks $c_i$ of $\cO H_i$ (so $B$ also covers the block $c\cong c_1 \otimes c_2$ of $\cO H$). If $C_G(H) \leq HZ(G)$ and $HZ(G)=\ker(G\to \Out(c))$, then:
\begin{enumerate}[label=(\Roman*)]
\item If $c_1$ is Morita equivalent to $\cO(A_4 \times Q_1)$, and $c_2$ is Morita equivalent to $\cO(A_4 \times Q_2)$ where $Q_1, Q_2 \in \{ 1, C_2 \}$, then $B$ is Morita equivalent to one of (i), (ii), (iv), (viii), (a).
\item If $c_1$ is Morita equivalent to $\cO(A_4 \times Q_1)$ and $c_2$ is Morita equivalent to $B_0(\cO(A_5 \times Q_2))$ where $Q_1, Q_2 \in \{ 1, C_2 \}$, then $B$ is Morita equivalent to one of (iii), (ix).
\item If each $c_i$ is Morita equivalent to $B_0(\cO(A_5 \times Q_i))$ where $Q_1, Q_2 \in \{ 1, C_2 \}$, then $B$ can only be Morita equivalent to $c$, i.e. (x).
\item If $c_1$ is Morita equivalent to $\cO((C_2)^3 \rtimes C_7)$ and $c_2$ is Morita equivalent to $\cO(A_4)$, then $B$ is Morita equivalent to one of (i), (ii), (iv), (vi), (viii), (xiii), (xvii), (xx), (xxiv), (a), (b).
\item If $c_1$ is Morita equivalent to $\cO((C_2)^3 \rtimes C_7)$ and $c_2$ is Morita equivalent to $B_0(\cO(A_5))$, then $B$ is Morita equivalent to one of (iii), (ix), (xiv), (xxv).
\item If $c_1$ is Morita equivalent to $\cO((C_2)^3 \rtimes (C_7 \rtimes C_3))$ and $c_2$ is Morita equivalent to $\cO(A_4)$, then $B$ is Morita equivalent to one of (xiii), (xvii), (xxiv), (b). 
\item If $c_1$ is Morita equivalent to $\cO((C_2)^3 \rtimes (C_7 \rtimes C_3))$ and $c_2$ is Morita equivalent to $B_0(\cO(A_5))$, then $B$ is Morita equivalent to one of (xiv), (xxv). \qedhere
\end{enumerate}
\end{proposition}
\begin{proof}
\sloppy We use Method \ref{clubsuit}, with the following improvement: since the action of $G$ by conjugation stabilises both $H_1$ and $H_2$, the image of $G/(HZ(G))$ through $\gamma\beta\alpha$ is always contained in the subgroup $\cT(c_1) \times \cT(c_2)$ of $\cT(c)$, which is then contained in $\Pic(c_1) \times \Pic(c_2) \leq \Pic(c)$. In each case, we denote the unique maximal subgroup of odd order of this image in $\Pic(c)$ as $T$, which controls all the possibilities that can occur for $G/(HZ(G))$. Let $N=HZ(G)$. In each case we pick a chain of normal subgroups $N \lhd N_1 \lhd \dots \lhd G$ and consider the corresponding block chain.
\begin{enumerate}[label=(\Roman*)]
\item In this case $\Pic(c_i)= S_3 \times Q_i$, so $T=C_3 \times C_3$. Let $\sigma : G \to S_3$ be the homomorphism given by the action of $G$ permuting the three simple modules of $c_1$. Note that either $G = \ker(\sigma)$ or $[G:\ker(\sigma)]=3$. Consider a chain of normal subgroups $\{N_i\}$ of length $t$ such that if $G \neq \ker(\sigma)$ then $N_{t-1}=\ker(\sigma)$, and the corresponding block chain $\{b_i\}$. Note that $t \leq 2$. Then since $N_1/N = C_3$ we have the following possibilities for the Morita equivalence class of $b_1$: 
\begin{enumerate}[label=(\arabic*)]
\item (ii), realised when $H=\PSL_3(7)^2 \times C_2$, and $N_1 \leq \ker(\sigma)$.
\item (ii), realised when $H=\PSL_3(7)^2 \times C_2$, and $N_1 \not \leq \ker(\sigma)$.
\item (iv), realised when   $N=\PSL_3(7)^2 \times C_2$ and $G=(\PSL_3(7)^2 \rtimes C_3) \times C_2$.
\end{enumerate}
In cases (2) and (3) the simple modules of $c_1$ are permuted transitively, so $G=N_1$ and $B=b_1$. In case (1) $N_1 \lhd G$, we have that $G/N_1 = (G/N)/(N_1/N) \cong C_3$, and from Proposition \ref{picard} there are three possible embeddings of $C_3$ in a subgroup of $\Pic(b_1)=S_3 \times \GL_3(2)$, which give the same possibilities as in cases (I.2,3,4) in Proposition \ref{a4ora5}.

\item In this case $\Pic(c_1)= S_3 \times Q_1$ and $\Pic(c_2)= C_2 \times Q_2$, so $T=C_3$. Then there is a unique possibility for the Morita equivalence class of $b_1 = B$: (iii), realised when $N=A_5 \times \PSL_3(7)$. 
\item This case is implied by the stronger result in Proposition \ref{a5a5c2}. Note that our technique also works in this special situation, since $\Pic(c_i) = C_2 \times Q_i$, a $2$-group, hence $G=H$.

\item In this case $\Pic(c_1)=C_7 \rtimes C_3$ and $\Pic(c_2)=S_3$. Hence $T=(C_7 \rtimes C_3) \times C_3$. Let $\sigma : G \to S_3$ be the homomorphism given by the action of $G$ permuting the three simple modules of $c_2$, and note that if $G \neq \ker(\sigma)$ then $[G:\ker(\sigma)]=3$. Consider a chain of normal subgroups $\{N_i\}$ of length $t$ where, if $G \neq \ker(\sigma)$, $N_{t-1}=\ker(\sigma)$, and the corresponding block chain. Note that $t \leq 3$. We have the following possibilities for $X_1 = N_1/N$ and the Morita equivalence class of $b_1$: 
	\begin{enumerate}[label=(\arabic*)]
	\item $|X_1|=7$, and $b_1$ is Morita equivalent to (ii), realised when $N_1=((C_2)^3 \rtimes 7^{1+2}_+) \times A_4$ and $N$ is a maximal subgroup of $G$ with index $7$.
	\item $|X_1|=3$, and $b_1$ is Morita equivalent to (xxiv), realised when $N=((C_2)^3 \rtimes C_7) \times A_4$.
	\item $|X_1|=3$, and $b_1$ is Morita equivalent to (xiii), realised when $N=((C_2)^3 \rtimes C_7) \times \PSL_3(7)$.
	\item $|X_1|=3$, and $b_1$ is Morita equivalent to (xx), realised when $N=((C_2)^3 \rtimes C_7) \times \PSL_3(7)$.
	\end{enumerate}
Note that in cases (3) and (4) the simple modules of $b_2$ are permuted by $N_1/N$, so $G=N_1$ and $B=b_1$. For the other two cases, we consider $b_2$ and $X_2=N_2/N_1$:
\begin{enumerate}[label=(\alph*)]
\item In this situation since $N_1 \lhd G$ we have that $G/N_1 \leq C_3 \times C_3$, and also $\Pic(b_1)=S_3 \times \GL_3(2)$. In particular the cases that can occur here are a subset of the cases already examined in Proposition \ref{a4ora5}, so we are done.
\item By inspection of all possible chains of normal subgroups $1 \lhd \dots \lhd G/N$ with prime indices, in this case $G/N \leq C_3 \times C_3$, and hence $G/N_1 \leq C_3$. From Proposition \ref{picard} there are three possible embeddings of $G/N_1$ in $\Pic(b_1)$, as the image of $G/N_1$ is contained in the unique maximal subgroup of odd order $C_3 \times C_3 \leq \Pic(b_1)$. So we have the following possibilities for the Morita equivalence class of $B$: 
\begin{itemize}
\item (xvii), realised when $N_1 =  ((C_2)^3 \rtimes (C_7 \rtimes C_3)) \times \PSL(3,7)$.
\item (xx), realised when $N_2=(((C_2)^3 \rtimes C_7)  \times \PSL_3(7))\rtimes 3^{1+2}_+$ and $N_1$ is a maximal subgroup of $N_2$ with index $3$.
\item (xiii), realised when $N_2=(C_2)^3 \rtimes (C_7 \rtimes 3^{1+2}_+) \times A_4$ and $N_1$ is a maximal subgroup of $N_2$ with index $3$.
\end{itemize}
Note that actually the last case cannot occur, as it implies $G= \ker(\sigma)$, but then the chain cannot have length two with $|G/N|=9$.
\end{enumerate}

\item In this case $\Pic(c_1)=C_7 \rtimes C_3$ and $\Pic(c_2)=C_2$, so $T=C_7 \rtimes C_3$. We have the following possibilities for $G/N$ and the Morita equivalence class of $B$.
\begin{enumerate}[label=(\arabic*)]
	\item \sloppy $|X_1|=7$, and $B$ is Morita equivalent to (iii), realised when $N_1=((C_2)^3 \rtimes 7^{1+2}_+) \times A_5$ and $N$ is a maximal subgroup of $N_1 \text{ with index $7$.}$
	\item $|X_1|=3$, and $B$ is Morita equivalent to (xxv), realised when $N=((C_2)^3 \rtimes C_7) \times A_5$.
	\item $|X_1|=21$, and $B$ is Morita equivalent to (ix). We are unable to realise this crossed product with a direct example of a chain of three groups, but if we suppose that $N_1/N=C_7$ then $b_1$ is as in case (1) above. Since from Proposition \ref{taugroups} $\Pic(b_1)$ admits a unique embedding of $C_3$ then (ix) is the unique possibility for the Morita equivalence class of $B$.
\end{enumerate}
\item In this case $\Pic(c_1) = C_3$ and $\Pic(c_2)=S_3$, so $T=C_3 \times C_3$. As before, let $\sigma : G \to S_3$ be the homomorphism given by the action of $G$ permuting the three simple modules of $c_2$, and note that either $G = \ker(\sigma)$ or $[G:\ker(\sigma)]=3$. Consider a chain of normal subgroups $\{N_i\}$ of length $t$ where, when $G \neq \ker(\sigma)$, $N_{t-1}=\ker(\sigma)$, and consider the corresponding block chain. Note that $t \leq 3$. We have the following possibilities for $X_1 = N_1/N = C_3$ and the Morita equivalence class of $b_1$: 
	\begin{enumerate}[label=(\arabic*)]
	\item (xiii), realised when $N_1=(C_2)^3 \rtimes (C_7 \rtimes 3^{1+2}_+) \times A_4$ and $N$ is a maximal subgroup of $N_1$ with index $3$.
	\item (xvii), realised when $N=((C_2)^3 \rtimes (C_7 \rtimes C_3)) \times \PSL_3(7)$.
	\item (xx), realised when $N_1= ((C_2)^3 \rtimes C_7) \times \PSL_3(7) \rtimes 3^{1+2}_+$ and $N$ is a maximal subgroup of $N_1$ with index $3$
	\end{enumerate}
In cases (2) and (3) the simple modules of $c_2$ are permuted by the action of $X_1$, so $G=N_1$ and $B=b_1$. In case (1) $N_1 \lhd G$ and clearly $G/N_1 \cong C_3$. From Proposition \ref{picard}, $\Pic(b_1)$ contains a unique subgroup isomorphic to $C_3 \times C_3$, and hence there are three possible embeddings of $G/N_1$ in $\Pic(b_1)$, which give the following possibilities for the Morita equivalence class of $B$:
\begin{itemize}
\item (vi), realised when $N_1 = ((C_2)^3 \times C_7) \times \PSL_3(7)$.
\item (b), realised as a crossed product when $N_1=((C_2)^3 \rtimes C_7)\times \PSL_3(7)$.
\item (xxiv), realised as a crossed product when $N_1=((C_2)^3 \times C_7) \times A_4$.
\end{itemize}
Note that the last case cannot occur, as it implies $G= \ker(\sigma)$, in which case the chain cannot have length two with $|G/N|=9$.

\item In this case $\Pic(c_1)=C_3$ and $\Pic(c_2)=C_2$, so $G/N \leq C_3$. Then $G=N_1$, and $b_1 = B$ is Morita equivalent to (xiv), realised when $G=(C_2)^3 \rtimes (C_7 \rtimes 3^{1+2}_+) \times A_5$ and $N$ is a maximal subgroup of $G$ with index $3$. \qedhere
\end{enumerate}
\end{proof}

The next lemma deals with situations in which the initial block is again a block of the direct product of two normal subgroups, but now one of the groups is fixed up to isomorphism.

\begin{proposition}\label{mixed}
Let $G$ be a finite group and $B$ be a quasiprimitive block of $\cO G$ with defect group $D= (C_2)^5$. Suppose that there are $H_1, H_2 \lhd G$, $H=H_1 \times H_2$ with $H \lhd G$ and $[G:H]$ odd, and suppose that $B$ covers $G$-stable blocks $c_i$ of $\cO H_i$, so $B$ also covers the block $c\cong c_1 \otimes c_2$ of $\cO H$. Suppose that $C_G(H) \leq HZ(G)$ and $HZ(G)=\ker(G\to \Out(c))$. \\ Suppose that $H_1$ is isomorphic to $\SL_2(8)$ or ${}^2G_2(3^{2m+1})$ for some $m \in \mathbb{N}$, and $c_1$ is the principal block, or that $H_1$ is isomorphic to $\operatorname{Co}_3$ and $c_1$ is the unique nonprincipal block with defect group $(C_2)^3$. Then either $B$ is Morita equivalent to $c$ or one of the following occurs:
\begin{enumerate}[label=(\Roman*)]
\item If $H_1\cong\SL_2(8)$ and $c_2$ is nilpotent then $B$ is Morita equivalent to one of (vii), (xv), (xix), (xxi), (xxviii), (c).
\item If $H_1\cong\SL_2(8)$ and $c_2$ is Morita equivalent to $\cO A_4$, then $B$ is Morita equivalent to one of (vii), (xv), (xix), (xxi), (xxviii), (c).
\item If $H_1\cong\SL_2(8)$ and $c_2$ is Morita equivalent to $B_0(\cO A_5)$, then $B$ is Morita equivalent to one of (xvi), (xxix).
\item If $H_1 \cong {}^2G_2(3^{2m+1})$ or $H_1 \cong \operatorname{Co}_3$ and $c_2$ is nilpotent, then $B$ is Morita equivalent to one of (xix), (xxviii).
\item If $H_1 \cong {}^2G_2(3^{2m+1})$ or $H_1 \cong \operatorname{Co}_3$, and $c_2$ is Morita equivalent to $\cO A_4$, then $B$ is Morita equivalent to one of (xix), (xxviii).
\item If $H_1 \cong {}^2G_2(3^{2m+1})$ or $H_1 \cong \operatorname{Co}_3$,  and $c_2$ is Morita equivalent to $B_0(\cO A_5)$, then $B$ is Morita equivalent to (xxix).
\item If $H_1 \cong J_1$ and $c_2$ is nilpotent then $B$ is Morita equivalent to one of (xviii), (xxvi).
\item If $H_1 \cong J_1$ and $c_2$ is Morita equivalent to $\cO A_4$ then $B$ is Morita equivalent to one of (xviii), (xxvi).
\item If $H_1 \cong J_1$ and $c_2$ is Morita equivalent to $B_0(\cO A_5)$ then $B$ is Morita equivalent to (xxvii).
\end{enumerate}
\end{proposition}
\begin{proof}
\noindent We use the same method as in Proposition \ref{twocomponents}, plus knowledge of the outer automorphism groups of the various possibilities for $H_1$.

\sloppy Some Picard groups are known from \cite{eali18}: from Proposition 5.3 and 5.4 in \cite{eali18}   $\Pic(B_0(\cO \SL_2(8))) = C_3$, and $\Pic(B_0(\cO \Aut(\SL_2(8)))) = \Pic(B_0 (\cO ({}^2G_2(q))))= C_3$. From \cite[1.5]{kmn11} and \cite[3.3]{okuyama}, the unique nonprincipal block of $\operatorname{Co}_3$ that has defect group $(C_2)^3$ and the principal blocks of $\Aut(\SL_2(8))$ and ${}^2G_2(q)$ are Morita equivalent.

In each case we consider a chain of normal subgroups $\{N_i\}$ of length $t$ with $N_0=H$ and $N_t = G$, and the corresponding block chain $\{b_i\}$. Since the first component is known up to isomorphism, it is enough to look at $\Out(H_1)$, instead of the Picard group of $c_1$. The Picard group of $c_2$ in each case controls the number of possibilities for $\cT(c_2)$ and, hence, for nontrivial crossed products. Again, as in Proposition \ref{twocomponents}, since the action by $G$ stabilises both $H_1$ and $H_2$, we only need to consider the subgroup $\gamma(\beta(\Out(H_1))) \times \cT(c_2)$ (see \ref{clubsuit}), which is controlled by $\Out(H_1) \times \Pic(c_2)$. In each case, we denote the unique maximal subgroup of odd order of this subgroup of $\Pic(c)$ as $T$.

\begin{enumerate}[label=(\Roman*)]
\item If $H_1=\SL_2(8)$ and $c_2$ is Morita equivalent to $\cO (C_2)^2$ then $T=C_3 \times C_3$. We distinguish two situations: if $G/(HZ(G)) \cong C_3$, then $B$ is a crossed product of $c$ with $C_3$, and there are three possible embeddings in $T$, which identify the following possibilities for $B$:
\begin{itemize}
	\item (xv), realised when   $H=\SL_2(8) \times (C_2)^2$.
	\item (xix), realised when   $H=\SL_2(8) \times (C_2)^2$.
	\item (xxi), realised when   $H=\SL_2(8) \times (C_2)^2$.
\end{itemize}

Otherwise, $G/(HZ(G)) \cong C_3 \times C_3 = T$. Then we consider the group   $H'= H_1' \times H_2 \lhd G$, where $H_1' = \Aut(\SL_2(8))$. Now we can repeat the argument above for $c' = c_1' \otimes c_2$ to obtain that the maximal subgroup of odd order of $\Pic(c')$ that we need to consider is $T=C_3 \times C_3$, since $\Pic(c_1')=C_3$. Since $B$ is a crossed product of $c'$ with $C_3$, again we have three possible embeddings of $G/H'Z(G)$ in $T'$ and the following possibilities:
\begin{itemize}
	\item (vii), realised when  $G=(\SL_2(8) \rtimes 3^{1+2}_+) \times (C_2)^2$ and $H'$ is a maximal subgroup of $G$ with index $3$.
	\item (xxviii), realised when   $H'=\Aut(\SL_2(8)) \times (C_2)^2$.
	\item (c), realised when $G=(\SL_2(8) \times (C_2)^2) \rtimes 3^{1+2}_+$ and $H'$ is a maximal subgroup of $G$ with index $3$.
\end{itemize}

\item If $H_1=\SL_2(8)$ and $c_2$ is Morita equivalent to $\cO A_4$ then $T=C_3 \times C_3$. We repeat the argument as in the previous case, distinguishing two situations: if $G/(HZ(G)) \cong C_3$, then there are three possibilities for $B$:
\begin{itemize}
	\item (xxviii), realised when $H=\SL_2(8) \times A_4$.
	\item (vii), realised when $H=\SL_2(8) \times \operatorname{PSL}_3(7)$.
	\item (c), realised when $G=(\SL_2(8) \times (C_2)^2) \rtimes 3^{1+2}_+$ and $H$ is a maximal subgroup of $G$ with index $3$.
\end{itemize}

\sloppy Otherwise, $G/H \cong C_3 \times C_3$ and we consider $H'= H_1' \times H_2 \lhd G$, where   $H_1' = \Aut(\SL_2(8))$, and repeat the argument for $c'$ to obtain that $T'=C_3 \times C_3$. Hence, there are again three possible embeddings of $G/(H'Z(G)) \cong C_3$ in $T'$, which determine the following possibilities for $B$:
\begin{itemize}
	\item (xv), realised when $G=(\SL_2(8) \rtimes 3^{1+2}_+) \times A_4$ and $H$ is a maximal subgroup of $G$ with index $3$.
	\item (xix), realised when $H=\Aut(SL_2(8)) \times \operatorname{PSL}_3(7)$.
	\item (xxi), realised when $H=((\Aut(\SL_2(8)) \times C_3) \times \operatorname{PSL}_3(7)) \rtimes C_3$.
\end{itemize}

\item\sloppy  If $H_1=\SL_2(8)$ and $c_2$ is Morita equivalent to $B_0(\cO A_5)$ then $T=C_3$. Then $t=1$, so $N_1 = G$ and there is only one nontrivial possibility for the embedding of $G/(HZ(G)) \cong C_3$ in $T$, which corresponds to $B$ being Morita equivalent to (xxix).
\end{enumerate}

	From \cite{g2inner} $\Out({}^2G_2(3^{2m+1})) = C_{2m+1}$, and the degrees of irreducible characters of the principal block $c_1$ occur with multiplicity $1$ or $2$, which implies that if $H_1={}^2G_2(3^{2m+1})$ then every automorphism of $H_1$ acts as an inner automorphism on $c_1$. Hence, in our situation $\beta(\alpha(\Out_\star(H_1))) = 1$. Moreover, $\Out(\operatorname{Co}_3) = 1$, and $\Out(J_1)=1$.  So we can limit our analysis to the subgroup $\cT(c_2)$ in the next cases.
\begin{enumerate}[label=(\Roman*)]
\setcounter{enumi}{3}
\item If $H_1 = {}^2G_2(q)$, for any $q=3^{2m+1}, m \in \mathbb{N}$ or $H_1 = \operatorname{Co}_3$, and $c_2$ is nilpotent, then $T=C_3$. Then $N_1 = G$ and there is only one nontrivial possibility for the embedding of $G/(HZ(G)) \cong C_3$ in $T$, which gives $B$ Morita equivalent to (xxviii), realised when $H={}^2G_2(q) \times (C_2)^2$.

\item If $H_1 = {}^2G_2(q)$, for any $q=3^{2m+1}, m \in \mathbb{N}$ or $H_1 = \operatorname{Co}_3$, and $c_2$ is Morita equivalent to $\cO A_4$, then $T=C_3$. Then $N_1 = G$ and there is only one nontrivial possibility for the embedding of $G/(HZ(G)) \cong C_3$ in $T$, which gives $B$ Morita equivalent to (xix), realised when $N={}^2G_2(q) \times \operatorname{PSL}_3(7)$.

\item  If $H_1 = {}^2G_2(q)$, for any $q=3^{2m+1}, m \in \mathbb{N}$ or $H_1 = \operatorname{Co}_3$, and $c_2$ is Morita equivalent to $B_0(\cO A_5)$, then $T=\{1\}$, so $G=HZ(G)$.

\item If $H_1 = J_1$ and $c_2$ is Morita equivalent to $\cO (C_2)^2$ then $\Pic(\cO(C_2)^2) = (C_2)^2 \rtimes S_3$, so $T=C_3$. Then $G=N_1$, $B=b_1$ and there is only one nontrivial possibility for the Morita equivalence class of $B$: (xxvi), realised when $H=J_1 \times (C_2)^2$.

\item If $H_1 = J_1$ and $c_2$ is Morita equivalent to $\cO A_4$ then $\Pic(\cO A_4) = S_3$, so $T=C_3$. Then $G=N_1$, $B=b_1$ and there is only one nontrivial possibility for the Morita equivalence class of $B$: (xviii), realised when $H=J_1 \times \PSL_3(7)$.

\item If $H_1 = J_1$ and $c_2$ is Morita equivalent to $B_0(\cO A_5)$ then $\Pic(c_2) = C_2$, so $T=\{1\}$, so $G=H$. \qedhere
\end{enumerate}
\end{proof}


In Method \ref{clubsuit}, we assume that $NZ(G)=\ker(G \to \Out(b))$, and the reason is that we can always reduce to this situation. Suppose we are in the situation of Method \ref{clubsuit}, except without the hypothesis $NZ(G)=\ker(G \to \Out(b))$. Recall the definition of $G[b]$ in Section 2: we define $G[b]_\cO$ as the group of elements acting as inner automorphisms on the block $b$ of $\cO N$. Then $G[b]_\cO \lhd G[b]$ via the canonical map $b \to b \otimes_\cO k$. We identify $G[b]_\cO$ with $\ker (G \to \Out(b))$. 

From Proposition \ref{gbstarlemma} there is a unique block $\hat{b}$ of $G[b]_\cO$ that is source algebra equivalent to $b$. So in general we can consider $G[b]$ and $\hat{b}$ instead of $N$ and $b$ and apply Method \ref{clubsuit} (since $\Pic(\hat{b})=\Pic(b)$ and $\cT(\hat{b})=\cT(b)$) to obtain all possible Morita equivalence classes for $B$. However, in Proposition \ref{twocomponents} and \ref{mixed} we have used the group structure to reduce to particular subgroups of $\Pic(b)$: to generalize these arguments we need to show that when the kernel $G[b]_\cO$ is nontrivial we can still apply the propositions.

\begin{lemma} \label{twocomponentscorollary}
Let $G$ be a finite group and $B$ be a quasiprimitive block of $\cO G$ with defect group $D= (C_2)^5$. Suppose that there are $N_1, N_2 \lhd G$, $N=N_1 \times N_2$ with $N \lhd G$ and $[G:N]$ of odd order, and suppose that $B$ covers $G$-stable blocks $c_i$ of $\cO N_i$, so $B$ also covers the block $c\cong c_1 \otimes c_2$ of $\cO N$. Suppose that $C_G(N) \leq NZ(G)$. Then, for each fixed pair of Morita equivalence classes of $c_1, c_2$ listed in cases I-VII of Proposition {\normalfont \ref{twocomponents}} and cases I-IX of Proposition {\normalfont \ref{mixed}}, the Morita equivalence class of $B$ is still among the ones listed in that same case in Proposition {\normalfont \ref{twocomponents}} or Proposition {\normalfont \ref{mixed}}.
\end{lemma}
\begin{proof}
We use the notation of Method \ref{clubsuit}. Let $G[c]_\cO=\ker(G \to \Out(c))$. Since each $N_i$ is a normal subgroup of $G$, $\alpha(G/N)$ is contained in $\Out_\star(N_1) \times \Out_\star(N_2)$. We consider the maps $\beta_i : \Out_\star(N_i) \to \Out(c_i)$ defined in the same way as $\beta$ in Method \ref{clubsuit}, and the map $\beta_1 \beta_2$, obtained by extending each $\beta_i$ to $\Out_\star(N_1) \times \Out_\star(N_2)$ such that $\beta_i|_{N_j}=\operatorname{id}_{N_j}$ when $i \neq j$. Since $c \cong c_1 \otimes c_2$, it is immediate that $\beta=\beta_1 \beta_2$. 
In particular then $\gamma(\beta(\alpha(G/N)))$ can be seen as a subgroup of $\Pic(c_1) \times \Pic(c_2)$, via injective maps $\gamma_i$ again defined as in Method \ref{clubsuit}. Note, however, that the map $\beta$ is not injective in general, as now $\ker(\beta) = G[c]_\cO$.

If we define $\hat{c}$ to be the unique block of $G[c]_\cO$ covered by $B$ and covering $c$ then from Lemma \ref{gbstarlemma} $\hat{c}$ is source algebra equivalent to $c$. Then in particular $\cT(c)=\cT(\hat{c})$, $\cE(c)=\cE(\hat{c})$ and $\Pic(c)=\Pic(\hat{c})$. Further, we can define maps $\hat{\alpha}$, $\hat{\beta}$ and $\hat{\gamma}$ replacing $N$ with $G[c]_\cO$ in the definitions of $\alpha, \beta, \gamma$. Since $N \leq G[c]_\cO$, $\hat{\alpha}$ is still injective, and $\hat{\beta}$ is also injective by definition of $G[c]_\cO$. Finally, $\hat{\gamma}$ is injective by definition. Further, repeating the argument above $\hat{\gamma}(\hat{\beta}(\hat{\alpha}(G/G[c]_\cO))) \leq \Pic(c_1) \times \Pic(c_2)$, and this time $\hat{\gamma}\hat{\beta}\hat{\alpha}$ is injective.

Hence, we can apply Proposition \ref{twocomponents} or Proposition \ref{mixed} (as appropriate), replacing $N$ and $c$ with $G[c]_\cO$ and $\hat{c}$, and obtain the same possibilities for the Morita equivalence class of $B$ since $\hat{\gamma}(\hat{\beta}(\hat{\alpha}(G/G[c]_\cO))) \leq \Pic(c_1) \times \Pic(c_2)$ from the discussion above, and hence we can replicate the proofs of both propositions.
\end{proof}

\begin{corollary} \label{isocliniccentralextensions}\quad \begin{enumerate} \setlength\itemsep{0em}
\item The nonprincipal blocks of $\cO(((C_2)^4 \rtimes 3^{1+2}_+) \times C_2)$ and $\cO(((C_2)^4 \rtimes 3^{1+2}_-) \times C_2)$ are all Morita equivalent. 
\item The nonprincipal blocks of $\cO((C_2)^5 \rtimes (C_7 \rtimes 3^{1+2}_+))$ and $\cO((C_2)^5 \rtimes (C_7 \rtimes 3^{1+2}_-))$ with $7$ simple modules are all Morita equivalent. 
\item The nonprincipal blocks of $\cO((\SL_2(8) \times (C_2)^2) \rtimes 3^{1+2}_+)$ and $\cO((\SL_2(8) \times (C_2)^2) \rtimes 3^{1+2}_-)$ with $7$ simple modules are all Morita equivalent. 
\end{enumerate}
\end{corollary}
\begin{proof}
This is immediate by considering in each case a maximal subgroup $N \lhd G$ of index $3$ and a block $b$ covered by $B$ of $\cO N$: \begin{enumerate}\itemsep0em
\item The claim immediately follows from \cite[3.3]{ea17}.
\item If we consider $N=((C_2)^3 \rtimes C_7) \times A_4 \times C_3$, it is a normal subgroup of both groups $G_1$ and $G_2$ listed, and it has only three blocks, all Morita equivalent to $\cO(((C_2)^3 \rtimes C_7) \times A_4)$. Since $N$ and $G_i$ satisfy the hypothesis of Proposition \ref{twocomponents}, there is only one possible Morita equivalence class for $B$ with $7$ simple modules.
\item If we consider $N=(\SL_2(8) \times A_4) \times C_3$, it is a normal subgroup of both groups $G_1$ and $G_2$ listed, and it has only three blocks with defect group $(C_2)^5$, all Morita equivalent to $B_0(\cO(\SL_2(8) \times A_4))$. Since $N$ and $G_i$ satisfy the hypothesis of Proposition \ref{mixed}, there is only one possible Morita equivalence class for $B$ with $7$ simple modules.\qedhere
\end{enumerate}
\end{proof}

\section{Blocks with defect group $(C_2)^5$}

First we classify the blocks with a normal defect group.

\begin{theorem} \label{normalD}
Let $B$ be a block of $\cO G$ where $G$ is a finite group, and suppose that $B$ has a normal defect group $D \cong (C_2)^5$. Then $B$ is Morita equivalent to $\cO (D \rtimes E)$ where $E$ is a subgroup of $\Aut(D)=\GL_5(2)$ of odd order, or to (a) or (b) as in Theorem \ref{maintheorem}.
\end{theorem}
\begin{proof} By \cite{kul85} a block with a normal defect group $D$ is Morita equivalent to a twisted group algebra $\cO_\alpha (D \rtimes E)$, where $E$ is the inertial quotient. Moreover, each possible $\alpha$ can be chosen as $\beta^{-1}$ where $ \beta \in H^2(E,k^\times)$ (see also \cite[6.14]{lin18}).

We have listed all possible inertial quotients in Proposition \ref{sambalecalcs}. Each group algebra $\cO (D\rtimes E)$ is also a block and, therefore, a representative of its class. 

It is a standard fact that twisted group algebras $\cO_\alpha H$ can be realised as blocks of the ordinary group algebra of a central extension $\hat{H}$ of $H$ by a $p'$-group \cite[10.5]{thevenaz}, so to produce examples of these Morita equivalence classes it is enough to look at odd central extensions of $D \rtimes E$, and hence at the Schur multiplier $M(E)$ of $E$. Now $M(E)$ is trivial whenever $E$ is cyclic (see \cite[1.2.10]{lin18}) and when $E=C_7 \rtimes C_3$ or $E=C_{31} \rtimes C_5$. When $E=C_3 \times C_3$ or $E=(C_7\rtimes C_3) \times C_3$, $M(E)=C_3$, giving two nontrivial possibilities for $\alpha$ in each case. 

When $E=C_3 \times C_3$, one of these corresponds to the Morita equivalence class whose representative is one of the nonprincipal blocks of $\cO ((C_2)^4 \rtimes 3^{1+2}_+) \times C_2$, where the center of the extraspecial group acts trivially. 

When $E=(C_7\rtimes C_3) \times C_3$, one of these corresponds to the Morita equivalence class whose representative is one of nonprincipal blocks of $\cO ((C_2)^5 \rtimes (C_7 \rtimes 3^{1+2}_+))$, where again the center acts trivially. 

From Corollary \ref{isocliniccentralextensions}, in each case choosing the other possibility for $\alpha$ (which corresponds to choosing $3^{1+2}_-$ instead of $3^{1+2}_+$) gives Morita equivalent blocks.\end{proof}

\noindent We prove our main result.

\begin{proof}[Proof of Theorem \ref{maintheorem}]
Let $B$ be a block of $\cO G$ for a finite group $G$ with defect group $D=(C_2)^5$, such that $B$ is not Morita equivalent to any of the blocks in the statement of the theorem and such that $([G:O_{2'}(Z(G))],|G|)$ is minimised in the lexicographic ordering.
First, we show that these hypotheses on $B$ imply three important facts: \begin{itemize}
\item[(I)] $B$ is quasiprimitive, that is, for any normal subgroup $N \lhd G$ any block of $\cO N$ covered by $B$ is $G$-stable.
In fact, let $N \lhd G$, and let $b$ be a block of $\cO N$ covered by $B$. We write $I_G(b)$ for the stabiliser of $b$ under conjugation by $G$. Then we can consider the Fong-Reynolds correspondent $B_I$ as in Theorem \ref{fong1}, the unique block of $I_G(b)$ covering $b$ and with Brauer correspondent $B$, that is Morita equivalent to $B$ and shares a defect group with it. By minimality, $I_G(b)=G$, and the same is true for any block of any normal subgroup of $G$.
\item[(II)] If there is a normal subgroup $N \lhd G$ such that $B$ covers a nilpotent block $b$ of $\cO N$, then $N \leq O_2(G)Z(G)$. This follows from minimality and quasiprimitivity, using Corollary \ref{kulshammerpuigcorollary}. In particular, note that this implies $O_{2'}(G) \leq Z(G)$.
\item[(III)] $G$ does not have any normal subgroup of index $2$: in fact, suppose by contradiction that there is $N \lhd G$ with index $2$. Let $b_N$ be the unique block of $\cO N$ covered by $B$. Then from \cite[5.3.5]{feit} $B$ is also the unique block covering $b_N$, so from \cite[15.1]{alp151} $D \cap N$ is a defect group of $b_N$ and $DN/N$ is a Sylow $2$-subgroup of $G/N$. Hence, $b_N$ has defect group $(C_2)^4$ and from \cite{ea17} we know all the possibilities for its Morita equivalence class and its inertial quotient. From the main theorem of \cite{kk96} $\overline{B}=B \otimes_\cO k$ is Morita equivalent to $(b_N \otimes_\cO k) \otimes kC_2$. Moreover, from Maschke's theorem (see \cite[3.3.2]{gor07}) there is an $E$-stable decomposition of $D$ as $(D\cap N) \times Q$ where $Q=C_2$. In particular then $Q \leq C_D(E)$. Then $E$ is cyclic of order $3$, $5$, $7$ or $15$ or it is one of $(C_3 \times C_3)$ or $(C_7 \rtimes C_3)_1$. Then by Proposition \ref{index2} $B$ is Morita equivalent to $b_N \otimes \cO C_2$, a contradiction since, from \cite{ea17}, all such blocks already appear in the list.
\end{itemize}
Recall the definition of the generalised Fitting subgroup $F^*(G)$ as in \cite{as00}: we write $E(G)$ for the layer of $G$, a normal subgroup that is the central product of the components of $G$ (the subnormal quasisimple subgroups). We write $F(G)$ for the Fitting subgroup. The generalised Fitting subgroup is defined as $F^*(G)=F(G)E(G) \lhd G$, and it is a central product of $F(G)$ and $E(G)$. A fundamental property is that $C_G(F^*(G))\leq F^*(G)$, so there is an injective group homomorphism from $G/F^*(G)$ to $\Out(F^*(G))$. 

In particular, in our situation, (II) implies that $F(G)=O_{2'}(Z(G))O_2(G)$, and (I) implies that there is a unique block of $F^*(G)$ covered by $B$. We call it $b^*$. By minimality, we can suppose that no subgroup of $O_{2'}(Z(G))$ is a direct factor of $G$. 

If $E(G)=1$, then $F^*(G)=O_2(G)Z(G)$ and since $D$ is a defect group   $O_2(G) \leq D$. Since $D$ is abelian: $$D \leq C_G(O_2(G)) = C_G(F^*(G)) \leq F^*(G) = O_2(G)Z(G)$$ 
Hence, $D=O_2(G)$, a contradiction by Proposition \ref{normalD}. So $E(G) \neq 1$.

Write $E(G) = L_1 * \dots * L_t$, where each $L_i$ is a component of $G$. We prove that $t \leq 2$. Let $b_E$ be the unique block of $E(G)$ covered by $B$, and let $D_E = D \cap E(G)$ be its defect group. For each $i$, let $b_i$ be the unique block of $L_i$ covered by $b_E$ (note that $E(G)=L_i C_{E(G)}(L_i)$, so $b_i$ is $E(G)$-stable). Let $D_i$ be its defect group. The components in $E(G)$ are permuted by the action of $G$ by conjugation. Consider a $G$-orbit of $u$ components $\{L_{i_1} * \dots * L_{i_u}\}$, and let $L$ be the subgroup generated by them. Note that all the components in the same orbit are isomorphic, with isomorphic blocks $b_{i_j}$ and defect groups $D_{i_j}$. Since $L \lhd G$, it has a unique block $b_L$ covered by $B$, that covers each $b_{i_j}$ as determined before. The block $b_L$ cannot be nilpotent because of (II). Hence, from Lemma \ref{centralproducts}, every block $b_{i_j}$ of $L_{i_j}$ is also not nilpotent, so the inertial quotient $E_{i_j}$ of $b_{i_j}$ is nontrivial. Then $[D_{i_j}, E_{i_j}] \geq (C_2)^2$, which implies that $D_{i_j} / (D_{i_j} \cap Z(L_{i_j})) \geq (C_2)^2$. Again from Lemma \ref{centralproducts} the defect group of $b_L$, $D_L = D \cap L$ is a central product of all $D_{i_j}$. Then $D_L \geq ((C_2)^2)^{u}$. Write the $G$-orbits of components in $E(G)$ as $L^j$, $j=1, \dots, w$, each the central product of $u_j$ isomorphic components. Then:
$$|D|= 2^5 = \sum_{i=1}^w |D_{L^i}| \geq (2^2)^{\sum_{j=1}^w u_j}$$
It follows that either there are two $G$-stable components $L_1$ and $L_2$, or there is a single orbit of components with $u_1=2$. In the latter case $G$ would have a normal subgroup of index $2$, the kernel of the homomorphism $G \to S_2$ given by the permutation of the components, a fact that contradicts (III). So $t \leq 2$, and if $t=2$ then both components are normal in $G$. 

\sloppy If $t=2$, so that $E(G)=L_1 * L_2$, then $(L_1 \cap L_2) \leq Z(E(G))$. In particular $O_2(Z(E(G))) \leq O_2(G) \leq D$, so if $O_2(Z(E(G)))$ is not trivial then it has order $2$, since for each component $|D_i/(D_i \cap Z(L_i)| \geq 2$. \\

Consider $O_2(G) \lhd G$. Then $C:=C_G(O_2(G))\lhd G$. Let $c$ be the unique block of $\cO C$ covered by $B$. Then it has defect group $D \cap C = D$, and $C_G(D) \leq C$ since $O_2(G) \leq D$. From \cite[15.1]{alp151} then $B$ is the unique block covering $c$, so in particular $[G:C]$ is odd since they share a defect group. Since $F^*(G) \leq C$, then $C/F^*(G) \hookrightarrow \operatorname{Out}(E(G))$, since $C$ centralises $F(G)=O_2(G)O_{2'}(Z(G))$, so $C/F^*(G)$ is solvable because of Schreier's conjecture (since $t \leq 2$). Moreover, $G/C$ is solvable since it has odd order. Hence since $(G/F^*(G))/(C/F^*(G)) = G/C$ is solvable and $C/F^*(G)$ is solvable then $G/F^*(G)$ is also solvable. From Lemma \ref{pcore}, $DF^*(G)/F^*(G)$ is a Sylow $2$-subgroup of $G/F^*(G)$.


Suppose that $t=1$. In this case $E(G)$ is a quasisimple group, so from Proposition \ref{quasisimpleblocks} we know all the possibilities that can occur: $b_E$ can be among blocks of $\SL_2(32)$, $\SL_2(16)$, $\SL_2(8)$, $J_1$, $\operatorname{Co}_3$, ${}^2G_2(3^{2n+1}) (n>0)$, $D_n(q)$, $E_7(q)$, or it can be nilpotent covered, or as in case (iv) of the Proposition.

First, note that if $L_1$ is as in case (iv) of Proposition \ref{quasisimpleblocks}, then $b_E$ is the block of a central extension by an elementary abelian $2$-group of a block $\overline{b_E}$ with defect group $C_2 \times C_2$. In particular, from Proposition 2.7 in \cite{ea17} $\overline{b_E}$ is source algebra equivalent to $\cO A_4$ or $B_0(\cO A_5)$, so from \cite[1.22]{sam14} $b_E$ has inertial quotient $(C_3)_1$. From Corollary 1.14 in \cite{pu01}\footnote{Note that whenever a block has an abelian defect group there are no essential pointed groups}, the first intermediate central extension by a $C_2$ of $\overline{b_E}$ is basic Morita equivalent to the principal block of a central extension by $C_2$ of $A_4$ or $A_5$, and the only central extension with abelian defect groups is the direct product. Now we consider a central extension by another $C_2$, and repeat the argument. A repeated application of Corollary 1.14 in \cite{pu01} yields that in this case $b_E$ is Morita equivalent to $\cO(A_4 \times (C_2)^y)$ or $B_0(\cO (A_5 \times (C_2)^y))$ for $y=0,1,2$ or $3$. 

Now we examine each possibility: 

If $E(G)=\SL_2(32)$, $b_E$ has defect group $D$, hence $D \leq F^*(G)$ and since $\Out(\SL_2(32))=C_5$ the only possibilities for $G$ are $\SL_2(32)$ and $\Aut(\SL_2(32))$, which are (xxiii), (xxxi) in the list.

If $E(G)=\SL_2(16)$, then again $\Out(F^*(G))=\Out(\SL_2(16))$ which has order $4$. In particular if $G \neq F^*(G)$ then $G$ has a normal subgroup of index $2$, which is a contradiction. So $O_2(G)=C_2$ and $G=F^*(G)=\SL_2(16) \times C_2$, which is (xii) in the list.

If $E(G)=\SL_2(8)$, ${}^2G_2(q)$, $\operatorname{Co}_3$ or $J_1$, then $|D\cap E(G)|=8$, and since $\Out(E(G))$ has odd order in each case we can assume that $D \leq F^*(G)$, so $O_2(G)=(C_2)^2$. We can then apply Proposition \ref{mixed} to obtain that the block is already in the list.

In every other case, from Lemma \ref{pcore} there is a chain of normal subgroups and blocks:
$$\begin{array}{ccccccccc}
F^*(G) &  \stackrel{\ell_o}{\lhd} &  N_o & \stackrel{\ell_e}{\lhd}& N_e & \stackrel{\ell_b}{\lhd} &  G[b_e] & \stackrel{\ell_g}{\lhd} & G \\
b^* &&  b_o && b_e && \hat{b} && B
\end{array}$$
where $N_o \leq G[b^*]$, $\ell_o, \ell_b, \ell_g$ are odd numbers and $\ell_e = [D:D\cap F^*(G)]$. Note that $\ell_e = 2$ is a contradiction to (III), and $\ell_e \leq 8$ otherwise $b^*$ is nilpotent, contradicting (II). 

We consider the corresponding block chain $b^*, b_o, b_e, \hat{b}, B$, noting that each block is covered by $B$ and, hence, is $G$-stable. From Lemma 3.1 in \cite{eali18c} we know that $D \leq G[b^*]$, and hence $N_e \leq G[b^*]$ since $DF^*(G)/F^*(G)$ is a Sylow $2$-subgroup of $G/F^*(G)$ from Lemma \ref{pcore}. Further, $b_o$ is source algebra equivalent to $b^*$ by Lemma \ref{gbstarlemma}. Therefore, $b_e$ is Morita equivalent to $b^* \otimes \cO (C_2)^{\log_2(\ell_e)}$, since if $D \leq F^*(G)$ then $\ell_e =1$ so $b_o = b_e$, and when $\ell_e=4 $ or $8$, $b_e$ satisfies the hypothesis of Proposition \ref{index2}, with an identical argument as in the proof of (III) above. Again from Lemma \ref{gbstarlemma}, $\hat{b}$ is source algebra equivalent to $b_e$. Note that, in light of the discussion preceding Lemma \ref{twocomponentscorollary}, the pair $(G[b_e],\hat{b})$ satisfies the hypothesis of Method \ref{clubsuit}.

Suppose that $\ell_e = 1$, so $D \leq F^*(G)$ and $b^*=b_e$. If $b_1$ is nilpotent covered then it is inertial, and hence so is $b^*$. Hence by Proposition \ref{nilpotentcovered} $B$ is inertial, a contradiction by Theorem \ref{normalD}. Otherwise $\hat{b}$ is Morita equivalent to (ii) or (iii), a contradiction by Proposition \ref{a4ora5}.

Now suppose that $\ell_e =4$. If $|D \cap L_1|=8$ then if $b^*=b_1$ is nilpotent covered then by Proposition \ref{quasisimpleblocks} it is of type $A_n$ or $E_6$, and hence by \cite[Table 5]{atlas} $G$ has a normal subgroup of index $2$, which is a contradiction to (III). If instead $b_1$ is Morita equivalent to $\cO (A_4 \times C_2)$ or $B_0(\cO(A_5 \times C_2))$, then $\hat{b}$ is as in the hypothesis of Proposition \ref{a4ora5}, which gives a contradiction. If $|D \cap L_1|=4$ then $O_2(G)=C_2$ and $b_1$ is Morita equivalent to $\cO (A_4 \times C_2)$ or $B_0(\cO(A_5 \times C_2))$, so $\hat{b}$ is again as in the hypothesis of Proposition \ref{a4ora5}, a contradiction.

Finally, suppose that $\ell_e = 8$. Then $|D \cap L_1|=4$ and $|O_2(G)|=1$, otherwise $b^*$ would be nilpotent contradicting (II). So $b^*$ is Morita equivalent to $\cO A_4$ or $B_0(\cO A_5)$, again a contradiction since then $\hat{b}$ satisfies the hypothesis of Proposition \ref{a4ora5}.\\

Now suppose that $t=2$, so $E(G)=L_1 * L_2$, where $L_i \lhd G$. Suppose that $D$ is not contained in $F^*(G)$. Since $|D \cap F^*(G)| \geq (2^2)^2$, then $|G/F^*(G)|_2=2$. Then $G$ has a normal subgroup of index $2$, contradicting (III). So we can assume that $D \leq F^*(G)$. We also assume without loss of generality that $|D \cap L_1| \geq |D \cap L_2|$.

We consider again the chain of normal subgroups and blocks, which now reduces to:
$$\begin{array}{cccccc}
F^*(G) & \stackrel{\ell_b}{\lhd} &  G[b^*] & \stackrel{\ell_g}{\lhd} & G \\
b^* &&  \hat{b} && B
\end{array}$$
since $D \leq F^*(G)$. Recall that $\hat{b}$ is source algebra equivalent to $b^*$ (Lemma \ref{gbstarlemma}). 

Let $L_\cap = L_1 \cap L_2$, and note that $L_\cap \leq Z(G)$. In particular the map $\alpha$ defined in \ref{clubsuit} with $N=F^*(G)$ factors through:
$$(G/Z(G)) / (F^*(G)/Z(G)) \to \Out_\star(L_1/L_\cap \times L_2/L_\cap) \leq \Out(L_1/Z(L_1)) \times \Out(L_2/Z(L_2))$$ 
where the last inclusion holds because each $L_i$ is normal in $G$ (see also \cite{bidwell} and then \cite[7.6]{sam14}). 

If $|D \cap L_1|=4$, then $O_2(G)=C_2$ and $L_\cap$, $Z(L_1)$ and $Z(L_2)$ have odd order. In this situation $b_1$ is source algebra equivalent to $\cO A_4$ or $B_0 (\cO A_5)$, and so is $b_2$. Then we can apply Lemma \ref{twocomponentscorollary} and Proposition \ref{twocomponents} (considering $H_1 = L_1 \times O_2(G)$) to obtain a contradiction. 

So without loss of generality we can suppose that $|D \cap L_1|=8$.

Suppose that $|D \cap L_2|=4$, so $b_2$ is source algebra equivalent to $\cO A_4$ or $B_0(\cO A_5)$. Then in particular $|O_2(G)|=1$, and $L_\cap$ has odd order. If $L_1 = \SL_2(8)$, ${}^2G_2(3^{2m+1})$, $\operatorname{Co}_3$ or $J_1$ then Lemma \ref{twocomponentscorollary} and Proposition \ref{mixed} give a contradiction. 

Otherwise, either $b_1$ is nilpotent covered, so basic Morita equivalent to a block with a normal defect group $(C_2)^3$, or $b_1$ is as in case (iii) of Proposition \ref{quasisimpleblocks}. In both cases, Lemma \ref{twocomponentscorollary} and Proposition \ref{twocomponents} give a contradiction. 

Then $|D \cap L_2|=8$, and $E(G)$ is a central product of $L_1, L_2$ such that $Z:= C_2 \leq L_\cap$. From Lemma \ref{centralproducts} we can suppose without loss of generality that $L_\cap = Z$. Note that $Z \leq O_2(G) \leq F^*(G)$.

Consider $\overline{b^*}$, the unique block of $E(G)/Z = L_1/Z \times L_2/Z$ dominated by $b^*$. Since $E(G)/Z = L_1/Z \times L_2/Z$, the direct product of two simple groups, Lemma \ref{centralproducts} and \cite{ea17} imply that $\overline{b^*}$ is source algebra equivalent to the principal block of one among $\cO(A_4 \times A_4)$, $\cO(A_4 \times A_5)$ or $\cO(A_5 \times A_5)$. In particular, both $b^*$ and $\overline{b^*}$ have inertial quotient isomorphic to $C_3 \times C_3$. Recall that source algebra equivalences are realised by trivial source bimodules, and hence are in particular basic Morita equivalences. Then the pair $(b^*,\overline{b^*})$ satisfies the hypothesis of Corollary 1.14 in \cite{pu01}, so $b^*$ is basic Morita equivalent to the principal block of a central extension by $C_2$ of $X \times Y$ where $X,Y \in \{A_4, A_5\}$: the only possibility for $b^*$ to have an abelian defect group is that the central extension is actually a direct product with $C_2$. Hence, there are three possibilities: \begin{itemize}
\item[(1)] That $b^*$ is basic Morita equivalent to $\cO (A_4 \times A_4 \times C_2)$. But then it is inertial, so by Proposition \ref{nilpotentcovered} the block $B$ is also inertial, which is a contradiction.
\end{itemize}
Since any central product of two perfect groups is perfect, we can use the same argument as in Lemma 7.6 in \cite{sam14} to show that $\Aut(E(G)) \leq \Aut(L_1/Z  \times L_2/Z)$. Then in particular $\Out_\star(E(G)) = \Out_\star(E(G)/Z) = \Out_\star(L_1) \times \Out_\star(L_2)$, since each $L_i$ is normal in $G$. Then by direct inspection of all the possibilities (listed in \cite{atlas}) $G/F^*(G)$ is a supersolvable group, as the only non-supersolvable outer automorphism group of a quasisimple group is given by triality in type $D_4(q)$, which however involves automorphisms with even order.
\begin{itemize}
\item[(2)] That $b^*$ is basic Morita equivalent to $B_0(\cO (A_5 \times A_5 \times C_2))$. We can apply Proposition \ref{a5a5c2} to show that $B$ is Morita equivalent to $b^*$ and hence to (x), a contradiction. 
\item[(3)] That $b^*$ is basic Morita equivalent to $B_0(\cO (A_4 \times A_5 \times C_2))$. Then $G/G[b^*]_\cO \leq \cT(b^*)$, and $\cT(b^*) \leq \cE (B_0(\cO (A_4 \times A_5 \times C_2)))$ since the equivalence is basic, so $G/G[b^*]_\cO \leq C_3$  by Proposition \ref{taugroups}. So using Method \ref{clubsuit} there is only one possible Morita equivalence class for $B$, the one of $B_0(\cO (A_5 \times (C_2)^3))$, realised when $F^*(G) = \PSL_3(7) \times A_5 \times C_2$. Then $B$ is Morita equivalent to (iii), a contradiction.
\end{itemize}

Therefore, in every possible case, we have a contradiction. To see that the classes are distinct it is enough to compute the Cartan matrices for each block, with the exception of (i) and (a) whose Cartan matrices are identical, for which we note that the number of irreducible characters $k(\cO((C_2)^4 \rtimes 3^{1+2}_+)b) \neq k(\cO D)$. 

The fact that the isomorphism class of an elementary abelian defect group is invariant under Morita equivalences is Corollary 1.6 in \cite{lin18p}.
\end{proof}
\begin{remark} \normalfont
A block Morita equivalent to (a) cannot be a principal block, since principal blocks with one simple module are nilpotent by \cite[6.13]{navarro}, but $k($(a)$)=16$. \\ Moreover, from the main theorem of \cite{kul85} together with Lemma 2.5 in \cite{puigusami93}, if a block of $\cO G$ with defect group $D$ and inertial quotient $E$ is principal then its Brauer correspondent is Morita equivalent to a non-twisted group algebra $\cO (D \rtimes E)$. Then in particular if Brou\'e's abelian defect group conjecture holds for blocks Morita equivalent to (b) or (c) then they cannot be principal blocks of any finite group, since the twisted group algebra (b) is distinguished from any non-twisted algebra of type $\cO(D\rtimes E)$ from its pair of values $k((b)), l((b))$.
\end{remark}

We can now improve Proposition \ref{sambalecalcs} with the exact values. As a corollary, we can immediately show that in most of our Morita equivalence classes the inertial quotient is preserved, as it is determined by the pair of values $(k(B),l(B))$. The only exceptions are two pairs of inertial quotients in which the numerical invariants coincide, where we give a partial result.
\begin{corollary} \label{truesambalecalcs}
Let $B$ be a block of $\cO G$ where $G$ is a finite group, with defect group $D=(C_2)^5$ and inertial quotient $E$. Then one of the following holds: 
\begin{itemize}\itemsep0em
\item $E= \{1\}$ and $k(B)=32$, $l(B)=1$.
\item $E \cong C_3$, $|C_D(E)|=8$ and $k(B)=32$, $l(B)=3$.
\item $E \cong C_3$, $|C_D(E)|=2$ and $k(B)=16$, $l(B)=3$.
\item $E \cong C_5$, and $k(B)=16$, $l(B)=5$.
\item $E \cong C_7$, and $k(B)=32$, $l(B)=7$.
\item $E \cong C_3 \times C_3$, and $k(B)=32$, $l(B)=9$ or $k(B)=16$, $l(B)=1$.
\item $E \cong C_{15}$, and $k(B)=32$, $l(B)=15$.
\item $E \cong C_7 \rtimes C_3$, $|C_D(E)|=4$, and $k(B)=32$, $l(B)=5$.
\item $E \cong C_{21}$, and $k(B)=32$, $l(B)=21$.
\item $E \cong C_7 \rtimes C_3$, $|C_D(E)|=1$ and $k(B)=16$, $l(B)=5$.
\item $E \cong C_{31}$, and $k(B)=32$, $l(B)=31$ .
\item $E \cong (C_7 \rtimes C_3) \times C_3$, and either $k(B)=32$, $l(B)=15$ or $k(B)=16$, $l(B)=7$.
\item $E \cong C_{31} \rtimes C_5$, and $k(B)=16$, $l(B)=11$.
\end{itemize}
In particular, Morita equivalent blocks have isomorphic inertial quotients with the same action on $D$, except possibly when the Morita equivalence class is (v), (xi) or (xii) in Theorem \ref{maintheorem}.
\end{corollary}
\begin{proof}
Proposition \ref{sambalecalcs} implies that $E$ is among the groups listed above, and that moreover for each fixed pair $(k(B), l(B))$ of blocks appearing in Theorem \ref{maintheorem} there is a single possible isomorphism class for $E$, with two exceptions: the pairs of values $(k(B),l(B))$ that occur in two distinct isomorphism classes for $E$ are $(16,5)$ and $(32,15)$: \begin{itemize}
\item Let $B$ be such that $k(B)=16$, $l(B)=5$. Then $B$ is Morita equivalent to one of (v), (xx), (xxi) in Theorem \ref{maintheorem}.
Suppose that $E \cong C_5$. Then from the main theorem of \cite{wa05} $B$ is perfectly isometric to (v). In particular then $B$ is Morita equivalent to (v), since this block is not perfectly isometric to (xx) or (xxi): in fact, a perfect isometry implies an isomorphism of the centers, but the Loewy lengths of the center, computed with Magma \cite{magma}, is respectively $4$ for (v) and $3$ for (xx) and (xxi). Then if $B$ is Morita equivalent to (xx) or (xxi), $e(B) \neq 5$, so every block in (xx) and (xxi) has inertial quotient $E=C_7 \rtimes C_3$ with $|C_D(E)|=1$.
\item Let $B$ be a block of $\cO G$ with $k(B)=32$, $l(B)=15$. Then $B$ is Morita equivalent to one of (xi), (xii), (xxiv), (xxv), (xxvi), (xxvii), (xxviii), (xxix) in Theorem \ref{maintheorem}.
Suppose that $E \cong C_{15}$. Then from the main theorem of \cite{wa05} $B$ is perfectly isometric to (xi). In particular then $B$ is Morita equivalent to (xi) or (xii), since (xi) is not perfectly isometric to any of the other possibilities for the Morita equivalence class: in fact, as before, a perfect isometry implies an isomorphism of the centers. In this case all the centers have Loewy length $3$ but, as computed with Magma \cite{magma}, the dimension of $J^2(Z($(xi)$))$ is $15$, while it is $21$ for each representative between (xxiv)-(xxix). In particular then if $B$ is Morita equivalent to any block in (xxiv)-(xxix) then $E \not\cong C_{15}$, so every block in (xxiv)-(xxix) has inertial quotient $E=(C_7 \rtimes C_3) \times C_3$. \qedhere
\end{itemize}
\end{proof}
At the moment we are unable to show that an arbitrary block with defect group $(C_2)^5$ and inertial quotient $(C_7 \rtimes C_3)_2$ is not Morita equivalent to (v), and that a block with defect group $(C_2)^5$, inertial quotient $(C_7 \rtimes C_3) \times C_3$ and $15$ simple modules is not Morita equivalent to (xi) or (xii), so these Morita equivalence classes could contain blocks with different inertial quotients. However, we want to point out that any example of such blocks would provide a counterexample to Brou\'e's abelian defect group conjecture (hence, in particular, it would need to be a nonprincipal block).

In \cite{ea17}, to show that Brou\'e's conjecture holds for blocks with defect group $D=(C_2)^4$ the author implicitly uses Proposition 6.10.10 in \cite{lin18}, and the fact that Alperin's weight conjecture holds for blocks with defect group $D$ by Proposition 13.4 in \cite{sam14}. At the moment, we do not have the analogous result for $(C_2)^5$, so Brou\'e's abelian defect group conjecture will be dealt with in a subsequent paper.

We highlight, however, that by Proposition \label{truesambalecalcs} and Theorem 4.36 in \cite{benroq} two principal blocks in the list of Theorem \ref{maintheorem} are derived equivalent if and only if they have the same inertial quotient with the same action on $D$, and the same number of simple modules $l(B)$. The nonprincipal blocks in that list cannot be derived equivalent to any other block in the list, because they have different values for the pair $(k(B),l(B))$. We do not prove here that the blocks labeled as (b) and (c) are derived equivalent.

\section{Harada's conjecture}
We denote as $G^0$ the set of $p$-regular elements of $G$. Harada's conjecture states that, for a block $B$ of a finite group $G$, if a nonempty $J \subseteq \Irr_K(B)$ is such that:
\begin{equation}
\sum_{\chi \in J} \chi(1)\chi(g)=0 \quad \forall g \in G \setminus G^0 \tag{$\dagger$}
\end{equation}
then $J=\Irr(B)$.

Note that if a proper subset $J$ of $\Irr(B)$ satisfies Harada's conjecture, then the complement $\Irr(B) \setminus J$ also does. Lemma 1 in \cite{sam18} shows that the property ($\dagger$) above is equivalent to the existence of a vector $v \in \mathbb{Z}^{l(B)}$ such that for every row $d_\chi$ of the decomposition matrix of $B$ it holds that $(d_\chi, v) = \chi(1)$ if $\chi \in J$ and $0$ if $\chi \not\in J$. This implies that, if a block satisfies Harada's conjecture, then any other block Morita equivalent to it also does. Then, in particular, it is enough to prove it for each representative determined in Theorem \ref{maintheorem}. 

\sloppy We have checked each class computationally using Magma \cite{magma}, computing $\sum_{\chi \in J} \chi(1)\chi(g)$ on all $2$-singular elements for each subset $J$ with less than $\lfloor \frac{k(B)}{2}\rfloor$ elements, and Harada's conjecture holds for each of the classes and, hence, for every block with defect group $(C_2)^5$.

\section*{Acknowledgements}
This paper is part of the work done by the author during his PhD at the University of Manchester, supported by a Manchester Research Scholar Award and a President's Doctoral Scholar Award.\\
The author thanks Charles Eaton, his Ph.D supervisor, for his constant support, for many helpful discussions and for his careful reading of the manuscript. The author also thanks Elliot McKernon, Michael Livesey and Claudio Marchi for many helpful discussions, Kai Ino for helping with the translation of \cite{wa00} from Japanese to English, and Benjamin Sambale for his valuable comments on the proof of Proposition \ref{gblocalnoncyclic} and for pointing out a mistake in an earlier draft of Section $5$.

\end{document}